\documentclass[a4paper, reqno]{amsart}
\usepackage[T1]{fontenc}
\usepackage{lmodern}
\usepackage{amsmath,amsthm,amssymb}
\usepackage[nobysame, alphabetic]{amsrefs}
\usepackage[margin=25mm]{geometry}

\usepackage{enumerate,microtype}
\usepackage{tikz}
\usetikzlibrary{cd, arrows}
\usepackage{tikz-cd}
\usepackage{spectralsequences}

\usepackage{url}

\usepackage{xcolor}

\definecolor{darkblue}{rgb}{0,0,0.6}

\usepackage[breaklinks, pdftex, ocgcolorlinks,colorlinks=true, citecolor=darkblue, filecolor=darkblue, linkcolor=darkblue, urlcolor=black]{hyperref}
\usepackage[capitalize,noabbrev]{cleveref}

\makeatletter
\newtheorem*{rep@theorem}{\rep@title}
\newcommand{\newreptheorem}[2]{%
	\newenvironment{rep#1}[1]{%
		\def\rep@title{#2 \ref{##1}}%
		\begin{rep@theorem}}%
		{\end{rep@theorem}}}
\makeatother

\newtheorem{proposition}{Proposition}[section]
\newtheorem{theorem}[proposition]{Theorem}
\newtheorem*{theorem*}{Theorem}
\newtheorem{corollary}[proposition]{Corollary}
\newtheorem{lemma}[proposition]{Lemma}

\newreptheorem{theorem}{Theorem}
\newreptheorem{lemma}{Lemma}
\newreptheorem{proposition}{Proposition}
\newreptheorem{corollary}{Corollary}

\newtheorem{thmx}{Theorem}

\crefname{thmx}{Theorem}{Theorems}

\theoremstyle{definition}
\newtheorem{definition}[proposition]{Definition}

\theoremstyle{remark}
\newtheorem{remark}[proposition]{Remark}
\newtheorem{hypothesis}[proposition]{Hypothesis}
\newtheorem{assumption}[proposition]{Assumption}
\newtheorem*{claim}{Claim}
\newtheorem*{remark*}{Remark}

\AddToHook{env/theorem/begin}{\crefalias{proposition}{theorem}}
\AddToHook{env/assumption/begin}{\crefalias{proposition}{assumption}}
\AddToHook{env/hypothesis/begin}{\crefalias{proposition}{hypothesis}}
\AddToHook{env/corollary/begin}{\crefalias{proposition}{corollary}}
\AddToHook{env/definition/begin}{\crefalias{proposition}{definition}}
\AddToHook{env/lemma/begin}{\crefalias{proposition}{lemma}}
\AddToHook{env/question/begin}{\crefalias{proposition}{question}}
\AddToHook{env/example/begin}{\crefalias{proposition}{example}}
\AddToHook{env/conjecture/begin}{\crefalias{proposition}{conjecture}}
\AddToHook{env/remark/begin}{\crefalias{proposition}{remark}}
\AddToHook{env/const/begin}{\crefalias{proposition}{const}}

\crefname{assumption}{Assumption}{Assumptions}
\crefname{hypothesis}{Hypothesis}{Hypotheses}
\crefname{theorem}{Theorem}{Theorems}
\crefname{proposition}{Proposition}{Propositions}
\crefname{corollary}{Corollary}{Corollaries}
\crefname{definition}{Definition}{Definitions}
\crefname{lemma}{Lemma}{Lemmas}
\crefname{question}{Question}{Questions}
\crefname{example}{Example}{Examples}
\crefname{conjecture}{Conjecture}{Conjectures}
\crefname{remark}{Remark}{Remarks}
\crefname{const}{Construction}{Constructions}

\numberwithin{equation}{section}

\newcommand{\defi}{\mathrm{def}}

\newcommand{\N}{\mathbb{N}}

\newcommand{\R}{\mathbb{R}}

\newcommand{\Z}{\mathbb{Z}}

\newcommand{\im}{\operatorname{Im}}

\DeclareMathOperator{\Tor}{Tor}

\DeclareMathOperator{\Top}{Top}
\DeclareMathOperator{\BO}{BO}
\DeclareMathOperator{\BSO}{BSO}
\DeclareMathOperator{\OO}{O}
\DeclareMathOperator{\SO}{SO}
\DeclareMathOperator{\BTop}{BTop}
\DeclareMathOperator{\Spin}{Spin}
\DeclareMathOperator{\BSpin}{BSpin}
\DeclareMathOperator{\spin}{Spin}
\DeclareMathOperator{\Pin}{Pin}
\DeclareMathOperator{\BPin}{BPin}
\DeclareMathOperator{\TopSpin}{TopSpin}

\newcommand{\ol}{\overline}
\newcommand{\wt}{\widetilde}
\newcommand{\wh}{\widehat}

\newcommand{\CP}{\mathbb{C}P}
\newcommand{\RP}{\mathbb{R}P}

\DeclareMathOperator{\Sq}{Sq}

\DeclareMathOperator{\Hom}{Hom}
\DeclareMathOperator{\ev}{ev}

\title{Stably exotic $4$-manifolds}

\author{Daniel Kasprowski}
\address{School of Mathematical Sciences\\ University of Southampton\\ United Kingdom}
\email{d.kasprowski@soton.ac.uk}

\author{Mark Powell}
\address{School of Mathematics and Statistics, University of Glasgow, United Kingdom}
\email{mark.powell@glasgow.ac.uk}

\begin{document}
	\begin{abstract}
		A pair of closed, smooth $4$-manifolds $M$ and $M'$ are \emph{stably exotic} if they are stably homeomorphic but not stably diffeomorphic, where stabilisation refers to connected sum with copies of $S^2 \times S^2$.
		Orientable stable exotica do not exist by a result of Gompf, but Kreck showed that nonorientable examples are plentiful.
		We investigate which values of the fundamental group $\pi$ and the first and second Stiefel-Whitney classes $w_1$ and $w_2$ admit stably exotic pairs, giving a complete description if~$H_5(\pi;\mathbb{Z})=0$. In particular we produce new stable exotica, and new settings in which they do not arise.
	\end{abstract}

	\def\subjclassname{\textup{2020} Mathematics Subject Classification}
	\expandafter\let\csname subjclassname@1991\endcsname=\subjclassname

	\maketitle

	\section{Introduction}
	Two closed, smooth $4$-manifolds $M$ and $M'$ are \emph{stably diffeomorphic} or \emph{stably homeomorphic} if there exists $k\in\N$ such that $M\#k(S^2\times S^2)$ and $M'\#k(S^2\times S^2)$ are diffeomorphic or homeomorphic, respectively. Gompf proved that orientable $4$-manifolds are stably diffeomorphic if and only if they are stably homeomorphic.
	In fact, he proved the following stronger result, generalising Wall's theorem~\cite{Wall-on-simply-conn-4mflds}.
	
	\begin{theorem*}[Gompf~\cite{Gompf84}]
		Let $M$ and $M'$ be closed, smooth, stably homeomorphic $4$-manifolds.
		\begin{enumerate}[(i)]
			\item Then  $M\#(S^2\mathrel{\wt\times} S^2)$ and $M'\#(S^2\mathrel{\wt\times} S^2)$ are stably diffeomorphic.
			\item Suppose that $M$ and $M'$ are orientable. Then $M$ and $M'$ are stably diffeomorphic.
		\end{enumerate}
	\end{theorem*}
	
	\begin{remark}\label{remark:totally-non-spin}
		If the universal cover $\wt M$ is non-spin, then $M\#(S^2\times S^2)$ is diffeomorphic to $M\#(S^2\mathrel{\wt\times} S^2)$~\cite{Wall-diffeomorphisms}*{Theorem~5.2}. Hence by Gompf's theorem $4$-manifolds $M$ with either $w_1(M)=0$ or $w_2(\wt{M}) \neq 0$ are stably diffeomorphic if and only if they are stably homeomorphic.
	\end{remark}

	We investigate the existence of \emph{stably exotic} $4$-manifolds, i.e.\ $4$-manifolds that are stably homeomorphic but are not stably diffeomorphic.
	By \cref{remark:totally-non-spin}, these are necessarily nonorientable and the universal cover must be spin.
Henceforth we restrict to such $4$-manifolds.
	
	\begin{hypothesis}\label{hyp}
		By a \emph{$4$-manifold} we will always mean a closed, smooth, nonorientable $4$-manifold with spin universal cover.
	\end{hypothesis}
	
	Kreck proved that there is an abundance of examples of stable exotica; see also constructions of Cappell--Shaneson~\cites{CS-new-four-mflds-bams,CS-new-four-mflds-annals} and Akbulut~\cites{Akbulut-fake-4-manifold,Akbulut-on-fake} for fundamental groups $\Z/2$ and $\Z$ respectively. Torres~\cite{Torres-JKTR} and Bais--Torres~\cite{Bais-Torres} constructed further stable exotica more recently.

	\begin{theorem*}[Kreck~\cite{Kreck84}*{Theorem~1}]
		Let $\pi$ be a finitely presented group and let $w\colon \pi \to \Z/2$ be a nontrivial homomorphism. Then there exist stably exotic $4$-manifolds $M$ and $M'$ with fundamental group~$\pi$ and orientation character $w$, of the form $M = D \# K3$ and $M' = D \# 11(S^2 \times S^2)$, for some 4-manifold $D$.
	\end{theorem*}

	In the case where the 2-Sylow subgroup is $\Z/2$, the stable homeomorphism and diffeomorphism classifications were also comprehensively studied by Debray~\cite{Debray}.
	The stably exotic $4$-manifolds appearing in all the previously cited works \cite{Kreck84,CS-new-four-mflds-bams,CS-new-four-mflds-annals,Akbulut-fake-4-manifold,Akbulut-on-fake,Torres-JKTR,Debray,Bais-Torres} satisfy $w_2(\nu_M)=w_1(\nu_M)^2$, where $\nu_M$ is the stable normal bundle of $M$.
	
	Our aim in this article is to understand, for a fixed fundamental group and orientation character, which values of the second Stiefel-Whitney class are realised by stably exotic $4$-manifolds.
	To formalise this, let~$M$ be a $4$-manifold as in \cref{hyp}, and let $c\colon M\to B\pi$ be a 2-connected map; in particular $\pi_1(M) \cong \pi$. Then there exist classes $w_1 \neq 0\in H^1(\pi;\Z/2)$ and $w_2\in H^2(\pi;\Z/2)$ with $c^*(w_i)=w_i(\nu_M)$; to see this for $w_2$ consider the exact sequence $H^2(\pi;\Z/2) \to H^2(M;\Z/2) \to H^2(\wt{M};\Z/2)$.
	The triple $(\pi,w_1,w_2)$ is called the \emph{normal 1-type} of~$M$.
Our first main result reads as follows.
	
	\begin{thmx}
		\label{thm:main}
		Let $\pi$ be a finitely presented group and let $w_1 \neq 0 \in H^1(\pi;\Z/2)$ and $w_2\in H^2(\pi;\Z/2)$.
		\begin{enumerate}[(i)]
			\item\label{it:main-1} If there exist  stably exotic $4$-manifolds with normal 1-type $(\pi,w_1,w_2)$, then $w_1^3=w_1w_2\in H^3(\pi;\Z/2)$.
			\item \emph{[Kreck]}\label{it:main-2} If $w_2 = w_1^2 \in H^2(\pi;\Z/2)$, there exist  stably exotic $4$-manifolds with normal 1-type $(\pi,w_1,w_2)$.
			\item\label{it:main-3} If $\pi$ has cohomological dimension at most $3$ and  $w_1^3=w_1w_2\in H^3(\pi;\Z/2)$, then there exist  stably exotic $4$-manifolds with normal 1-type $(\pi,w_1,w_2)$.
		\end{enumerate}
	\end{thmx}

The example in the following remark shows that $w_1^3 = w_1w_2$ does not imply $w_2=w_1^2$ in general, even for cohomological dimension 3 groups, and hence
\cref{thm:main}~\eqref{it:main-3} produces new stable exotica that did not appear in any of the previous work cited above, and \eqref{it:main-2} does not imply~\eqref{it:main-3}.


\begin{remark}\label{remark-showing-ii-does-not-imply-iii}
Note that $w_2=w_1^2$ implies $w_1^3 = w_1w_2$.  The converse does not hold in general. 
	For a (cohomological dimension two) example where  $w_1^3 = w_1w_2$ but  $w_2\neq w_1^2$, let $\pi \cong \Z^2$.  Then $H^*(B\Z^2;\Z/2) \cong \wedge^2 (\Z/2)$. Let $w_1 = e_1$ and let $w_2 = e_1 \wedge e_2$. Then $w_1^3 = e_1 \wedge e_1 \wedge e_1 = 0 = e_1 \wedge e_1 \wedge e_2 = w_1 w_2$ but $w_1^2 = e_1 \wedge e_1 = 0 \neq e_1 \wedge e_2 = w_2$. By \cref{thm:main}~\eqref{it:main-3}, it follows that there exist stably exotic $4$-manifolds with this normal 1-type. Note that since $w_2$ neither equals $w_1^2$ nor vanishes, these stably exotic $4$-manifolds admit neither a $\Pin^+$ nor a $\Pin^-$ structure.
\end{remark}

The example in \cref{remark-showing-ii-does-not-imply-iii}, with non-cyclic fundamental group and no $\Pin$ structure, contrasts with the following corollary.


\begin{corollary}\label{cor:pin-structures}
	Let $M$ and  $M'$ be a stably exotic pair with cyclic fundamental group. Then $M$ and $M'$ admit a tangential $\Pin^+$-structure, i.e.\ $w_2(TM)=0$.
\end{corollary}

\begin{proof}
	By \cref{remark:totally-non-spin}, $\wt M$ is necessarily spin. Hence, to show $w_2(TM)=w_2(\nu M)+w_1(\nu M)^2=0$, it suffices to show that the normal $1$-type $(\pi,w_1,w_2)$ satisfies $w_2=w_1^2$. This holds for $\pi\cong \Z$ because~$H^2(\Z;\Z/2) =0$.
	
	Now let $\pi$ be a finite cyclic group, necessarily of even order since $M$ and $M'$ are stably exotic and hence nonorientable.
It follows from \cref{thm:main}~\eqref{it:main-1} that $w_1w_2=w_1^3 \in H^3(\pi;\Z/2) \cong \Z/2$.
For such $\pi$, cupping with the nontrivial element $w_1\in H^1(\pi;\Z/2)$ induces an isomorphism $H^2(\pi;\Z/2)\to H^3(\pi;\Z/2)$, and hence $w_1w_2=w_1^3$ implies $w_2=w_1^2$,  as desired.
\end{proof}

\subsection{The case \texorpdfstring{$w_1^3=w_1w_2$}{w13=w1w2}}
We describe a secondary obstruction to the existence of stably exotic 4-manifolds with normal $1$-type $(\pi,w_1,w_2)$ such that $w_1^3=w_1w_2$.
For this we consider the fibration sequence
\[F \xrightarrow{u} K(\Z/2,2)\times K(\Z/2,1) \xrightarrow{\iota_1^3+\iota_1\iota_2}K(\Z/2,3),\]
where $\iota_k \in H^k(K(\Z/2,k);\Z/2) \cong \Z/2$, for $k=1,2$, are the generators, and $F$ is by definition the homotopy fibre of the map  $\iota_1^3+\iota_1\iota_2$.
We write
\[\Sq^2_{w_1,w_2} := \Sq^2(-)+ \big(w_1 \cup \Sq^1(-)\big) + \big(w_2 \cup - \big) \colon H^2(\pi,\Z/2)\to H^4(\pi;\Z/2).\]

\begin{thmx}
	\label{thm:B}
\leavevmode
	\begin{enumerate}[(i)]
		\item\label{it:B-1} The space $F$ is 3-coconnected with $\pi_1(F)\cong \Z/2$, $\pi_2(F)\cong (\Z/2)[\Z/2]$, and trivial $k$-invariant.
		\item\label{it:B-2}
	Let $p\colon \wt{F}\to F$ be the universal cover and let $\{x_1,x_2\}$ be the basis of $H^2(\wt{F};\Z/2) \cong \pi_2(F)$ such that $T^*x_1=x_2$, where $T$ is the deck transformation.
There exists a unique class $\mathfrak{o}\in H^4(F;\Z/2)$ such that $p^*\mathfrak{o}=x_1x_2$ and $s^*\mathfrak{o}=0$ for any section $s\colon K(\Z/2,1)\to F$.
		\item\label{it:B-3} If there exist stably exotic $4$-manifolds with normal 1-type $(\pi,w_1,w_2)$, then \[[f^*\mathfrak{o}] =0 \in H^4(\pi;\Z/2)/\im \Sq^2_{w_1,w_2},\]
		for any lift $f\colon B\pi\to F$ of $w_2\times w_1\colon B\pi\to K(\Z/2,2)\times K(\Z/2,1)$.
		\item \label{it:B-4} If $H_5(\pi;\Z)=0$, then the converse to \eqref{it:B-3} holds, i.e.\ if $w_1^3 = w_1w_2$ and
		\[[f^*\mathfrak{o}] = 0\in H^4(\pi;\Z/2)/\im \Sq^2_{w_1,w_2}\]
        for some lift $f\colon B\pi\to F$ of $w_2\times w_1\colon B\pi\to K(\Z/2,2)\times K(\Z/2,1)$,
		then there exist stably exotic $4$-manifolds with this normal 1-type.
	\end{enumerate}
\end{thmx}

For \cref{thm:B} \eqref{it:B-3} note that the existence of a lift $f$ follows from \cref{thm:main}~\eqref{it:main-1}.

\begin{remark}
	We will show in \cref{prop:choice-of-lift}, that the class $[f^*\mathfrak{o}]\in H^4(\pi;\Z/2)/\im \Sq^2_{w_1,w_2}$
	is independent of the choice of lift $f$.
\end{remark}

\begin{remark}\label{remark-in-intro-example}
	   In \cref{prop:ex1.5} we will use \cref{thm:B} \eqref{it:B-3} to give an example of a normal $1$-type $(\pi,w_1,w_2)$ such that $w_1^3=w_1w_2$ but $[f^*\mathfrak{o}] \neq 0$, and  no stably exotic $4$-manifolds with normal $1$-type $(\pi,w_1,w_2)$ exist. In particular, this shows that the converse of \cref{thm:main} \eqref{it:main-1} does not hold.
\end{remark}

	If $w_1^3=w_1w_2$ and $0=[f^*\mathfrak{o}]$, there is one further obstruction to the existence of  stably exotic $4$-manifolds with normal 1-type $(\pi,w_1,w_2)$. To describe it, we have to first give a brief overview of the proofs of \cref{thm:main,thm:B}.
	
	By \cref{prop:strategy}, the existence of stably exotic $4$-manifolds with normal 1-type $(\pi,w_1,w_2)$ depends on the bordism class of the $K3$ surface in $\Omega_4(\xi_\pi)$ as defined in \eqref{eqn:pullback-defining-B-xi}. There is a James spectral sequence converging to $\Omega_4(\xi_\pi)$ and the $K3$ surface generates the term $E_{0,4}^2=H_0(\pi;\Omega_4^{\spin})$. Since the action of $\pi$ on $\Omega_4^{\spin}$ is given by the nontrivial $w_1$, $E_{0,4}^2\cong \Z/2$. Hence $0=[K3]\in \Omega_4(\xi_\pi)$ if and only if there is a nontrivial differential with codomain $E_{0,4}^k$ for some $2\leq k<\infty$. The possible nontrivial differentials occur on the $3,4$ and $5$ page. In \cref{section:proof-thm-A}, we will show that the $d_3$ differential $E_{3,2}^3\cong H_3(\pi,\Z/2)/\im d_2\to E_{0,4}^3$ is dual to $w_1^3+w_1w_2$ showing \cref{thm:main} \eqref{it:main-1} and \eqref{it:main-3}. In \cref{section:determining-d-4-diff}, we will show that if $w_1^3=w_1w_2$, the $d_4$ differential is dual to
	\[[f^*\mathfrak{o}]\in H^4(\pi;\Z/2)/\im \Sq^2_{w_1,w_2}\]
	showing \cref{thm:B} \eqref{it:B-3}. We will end the article by showing that also the $d_5$ differential is nontrivial in general. For this we show that $0=[K3]\in \Omega_4(\xi_{G})$, where $G$ is by definition the homotopy fibre of $F\xrightarrow{\mathfrak{o}}K(\Z/2,4)$.

\subsection{Examples with small Euler characteristic}
	
	While \cref{thm:main,thm:B} abstractly show the existence of stably exotic pairs, the following theorem gives a realisation result with small Euler characteristic. 
	For a finitely presented group $G$, let $\defi(G)$ be the deficiency of $G$, that is the maximum over all presentations of the number of generators minus the number of relators. In the following result is a generalisation of  
\cite{torres}*{Theorem A}, and the proof is similar.  Torres used an $\eta$ invariant in his proof, which needs a tangential $\Pin^+$ structure for its definition. We use instead the connection between stable exotica and the $K3$ surface. 
As discussed in \cref{remark-showing-ii-does-not-imply-iii}, we find stable exotica $M$ and $M'$ that do not admit any $\Pin^{\pm}$ structure, so our result goes further. 

	\begin{thmx}
		\label{thm:C}
		Let $(G,w_1,w_2)$ be such that there exist stably exotic $4$-manifolds with normal $1$-type $(G,w_1,w_2)$. Then there exist $4$-manifolds $M,M'$ such that:
		\begin{enumerate}[(i)]
			\item\label{it:smallex1} $M$ and $M'$ have normal $1$-type $(G,w_1,w_2)$ and Euler characteristic $4-2\defi(G)$;
			\item\label{it:smallex2} $M$ and $M'$ are homeomorphic but not stably diffeomorphic;
			\item\label{it:smallex4} $M'$ is obtained from $M$ by performing a Gluck twist along a smoothly embedded sphere; 
			\item\label{it:smallex3} $M\#\CP^2$ and $M'\#\CP^2$ are diffeomorphic;
			\item\label{it:smallex5} the orientation double covers $\wh M$ and $\wh M'$ are diffeomorphic.
		\end{enumerate}
	\end{thmx}

\begin{remark}
Euler characteristic $4-2\defi(G)$ is the smallest known; we do not know whether it is optimal. 
\end{remark}

\subsection*{Organisation of the paper}
In \cref{section:pin-structures} we recall some facts we will need on Pin structures and on Pin bordism.
In \cref{section:background} we recall modified surgery theory, as it pertains to stable exotica. 
In \cref{sec:applications-modified} we apply modified surgery to the study of stable exotica, in particular laying out the strategy for the proofs of the \cref{section:proof-thm-A,sec:proof-thm-B}. We also prove \cref{thm:C} here. 
Then we prove \cref{thm:main} in \cref{section:proof-thm-A}. 
In \cref{sec:proof-thm-B} we define the universal triple $(F,v_1,v_2)$ satisfying $v_1^3 = v_1v_2$, and analyse its fourth bordism group, showing that it admits a nontrivial $d_4$ differential. In \cref{section:determining-d-4-diff} we define the class $\mathfrak{o}$, use it to describe the $d_4$ differential, and prove \cref{thm:B}. In \cref{section:determining-d-5-diff} we define the universal space $G$ with vanishing $d_3$ and $d_4$ differentials, and show that the $d_5$ differential is nontrivial for this space.

\subsection*{Acknowledgements}
We are grateful to Ulrich Bunke, Daniel Galvin, Markus Land, Peter Teichner, Rafael Torres, and Simona Vesel\'{a} for helpful discussions.
Mark Powell was partially supported by EPSRC New Investigator grant EP/T028335/2.


\section{Pin structures and Pin bordism}\label{section:pin-structures}

We briefly recall $\Pin^{\pm}$ structures on vector bundles, and the associated $4$-dimensional bordism groups.
We refer to Kirby-Taylor~\cite{KT} for more information.
Let $k \geq 3$ (the discussion can be extended to $k <3$, but we do not need this).
The groups $\Pin^{\pm}(k)$ are double covers of $\OO(k)$, and both fit into a central extension $0 \to \Z/2 \to \Pin^{\pm}(k) \to \OO(k) \to 1$. See \cite{KT}*{\S 1} for the precise definitions of $\Pin^{\pm}(k)$.

Let $\zeta = (E \xrightarrow{p} B)$ be a rank $k$ vector bundle over a CW complex $B$.  A \emph{$\Pin^{\pm}(k)$ structure} on $\zeta$ is a reduction of the structure group from $\OO(k)$ to $\Pin^{\pm}(k)$, or equivalently a lift of the classifying map as follows:
\[\begin{tikzcd}
	& \BPin^{\pm}(k) \ar[d] \\ B \ar[ur,dashed] \ar[r,"\zeta"] & \BO(k).
\end{tikzcd}\]
There are corresponding stable spaces $\Pin^{\pm}$ and $\BPin^{\pm}$, and a $\Pin^{\pm}$ structure on a stable vector bundle $\eta \colon B \to \BO$ is a lift of $\eta$ to $\BPin^{\pm}$ along $\BPin^{\pm}\to \BO$.

A $\Pin^+(k)$ structure on $\zeta$ is equivalent to the data of a spin structure on $\zeta \oplus 3 \det (\zeta)$, and $\zeta$ admits a $\Pin^+(k)$ structure if and only if $w_2(\zeta)=0 \in H^2(B;\Z/2)$.
Similarly, a $\Pin^-(k)$ structure on $\zeta$ is equivalent to a spin structure on $\zeta \oplus  \det (\zeta)$, and $\zeta$ admits a $\Pin^-(k)$ structure if and only if $w_2(\zeta) + w_1(\zeta)^2=0 \in H^2(B;\Z/2)$.
It follows that a spin manifold admits both a $\Pin^+$ and a $\Pin^-$ structure.
We will need the following two propositions on $\Pin^{\pm}$-structures.

\begin{proposition}\label{prop:pin-structures}
	A compact manifold $W$ of dimension at least three admits a tangential $\Pin^{\pm}$ structure if and only if the stable normal bundle $\nu_W$ of an embedding $W \hookrightarrow \R^n$, $n$ large, admits a $\Pin^{\mp}$ structure.
\end{proposition}

\begin{proof}
	Since $TW\oplus \nu_W$ is trivial for any embedding $W\hookrightarrow \R^n$, the Whitney sum formula implies that $w_1(TW)=w_1(\nu_W)$, $w_2(TW) = w_1(\nu_W)^2 + w_2(\nu_W)$, and $w_2(\nu_W) = w_1(TW)^2 + w_2(TW)$.
\end{proof}

Using $\Pin^{\pm}$ structures on tangent bundles of $4$-manifolds and bordisms between them, there are corresponding bordism groups $\Omega_4^{\Pin^{\pm}}$.

\begin{proposition}[\cite{ABP69}, {\cite{KT}*{Theorem~5.2~and~Lemma~5.3}}]\label{prop:pin-4-bordism-groups}
	The $4$-dimensional $\Pin^{\pm}$ bordism groups are $\Omega_4^{\Pin^-} = 0$ and $\Omega_4^{\Pin^+} \cong \Z/16$. In the latter group $[K3]$ is order two, corresponding to $8 \in \Z/16$.
\end{proposition}

\section{Modified surgery theory}\label{section:background}

In this section we recall the facts and tools that we will need from the stable classification part of modified surgery theory, introduced by Kreck~\cite{surgeryandduality}.
We will only consider  spaces that have the homotopy type of a CW complex.

\begin{definition}
	Let $M$ be a closed, smooth $4$-manifold. A \emph{normal 1-type} of $M$ is
	a fibration over $\BO$, denoted by $\xi\colon B \to \BO$, through which a map representing the stable
	normal bundle $\nu_M \colon M \to \BO$ factors as follows:
	\[\begin{tikzcd}
		&B\ar[d,"\xi"]\\
		M\ar[r,"\nu_M"']\ar[ur,"\ol{\nu}_M"]&\BO
	\end{tikzcd}\]
	with $\ol\nu_M$ a 2-connected map and $\xi$ a 2-coconnected map. A choice of $\ol\nu_M$ is called
	a \emph{normal 1-smoothing} of $M$.
\end{definition}
All the normal $1$-types of $M$ are fibre homotopy equivalent to one another~\cite{Baues-obstruction-theory}. The fibre homotopy type determines and is determined by $(\pi,w_1,w_2)$, which justifies our alternative equivalent definition of the normal 1-type given in the introduction.

\begin{theorem}[{\cite{surgeryandduality}*{Theorem~C}, \cite{Crowley-Sixt}*{Lemma~2.3}}]\label{thm:kreck-thm-C-smooth}
	Two closed, smooth $4$-manifolds with fibre homotopy equivalent normal 1-types are stably diffeomorphic if and only if they have the same Euler characteristics and  admit bordant normal 1-smoothings.
\end{theorem}

It follows that to understand the stable diffeomorphism classes of $4$-manifolds with normal $1$-type $\xi$, one has to compute the bordism group $\Omega_4(\xi)$ of closed $4$-manifolds with a map $M \to B$ lifting the stable normal bundle $M \to \BO$ along $\xi$.

Forgetting that $M$ admits a smooth structure, we can consider analogous fibrations over $\BTop$.

\begin{definition}
	We call a 2-coconnected fibration $\xi^{\Top} \colon B^{\Top} \to \BTop$ a \emph{topological normal 1-type} of $M$ if the topological stable normal bundle $\nu^{\Top}_M \colon M \to \BTop$ factors as follows:
	\[\begin{tikzcd}[column sep=large]
		&B^{\Top} \ar[d,"\xi^{\Top}"]\\
		M\ar[r,"{\nu^{\Top}_M}"']\ar[ur,"{\ol{\nu}^{\,\Top}_M}"]&\BTop
	\end{tikzcd}\]
	with $\ol{\nu}^{\Top}_M$ a 2-connected map. We call a choice of $\ol{\nu}^{\Top}_M$
	a \emph{topological normal 1-smoothing} of $M$.
\end{definition}

The analogue of~\cref{thm:kreck-thm-C-smooth} also holds in the topological category.

\begin{theorem}[Kreck, Crowley--Sixt]\label{thm:kreck-thm-C-top}
	Two closed $4$-manifolds with fibre homotopy equivalent topological normal 1-types are stably homeomorphic if and only if they have the same Euler characteristics and admit bordant topological normal 1-smoothings.
\end{theorem}

Thus to understand the difference between stable homeomorphism and stable diffeomorphism, one can consider the forgetful map $\Omega_4(\xi)\to\Omega_4(\xi^{\Top})$ between the bordism groups of the smooth and topological normal 1-types associated to a fixed triple $(\pi,w_1,w_2)$.

The main examples of normal 1-types will be as in the next definition.

\begin{definition}\label{defn:key-pullback}
	Let $Y$ be a space together with a fibration $(w_1,w_2) \colon Y\xrightarrow{} K(\Z/2,1) \times K(\Z/2,2)$.
    \begin{enumerate}
        \item
    Then we write $\xi_Y\colon B_Y\to \BO$ for the pullback
	\[  \begin{tikzcd}[column sep = large]
		B_Y \ar[r] \ar[d,"\xi_Y"] & Y \ar[d,"{(w_1,w_2)}"] \\ \BO \ar[r,"P_2(-)"] & K(\Z/2,1) \times K(\Z/2,2).
	\end{tikzcd}\]
\item We let $\xi_Y^{\Top} \colon B_Y^{\Top} \to \BTop$ be the analogous pullback in the topological category.
\item When $Y= B\pi$, for a group $\pi$, we write $\xi_\pi \colon B_\pi \to \BO$ and $\xi_\pi^{\Top} \colon B_\pi^{\Top} \to \BTop$.
    \end{enumerate}
\end{definition}

As in \cref{hyp}, recall that we consider closed, nonorientable, smooth $4$-manifolds~$M$ with spin universal cover.

\begin{lemma}[\cite{surgeryandduality}*{p.~713}, \cite{teichner-phd}*{Theorem~2.2.1~b)(I)}]\label{lemma:def-key-pullback-gives-correct-1-type}
    Taking $Y:= B\pi_1(M)$ and $w_1 \neq 0 \in H^1(\pi;\Z/2)$ and $w_2 \in H^2(\pi;\Z/2)$ as in the introduction, \cref{defn:key-pullback} gives the smooth and topological normal 1-types of~$M$.
\end{lemma}

\begin{remark}\label{remark:James-SS}
The bordism groups $\Omega_4(\xi_Y)$ and $\Omega_4(\xi^{\Top}_Y)$ can be computed using \emph{James spectral sequences}.
They have the form  \[E^2_{p,q}=H_p(Y;\Omega_q^{\Spin}) \, \Rightarrow \, \Omega_*(\xi_Y)\] and
\[E^{2,\Top}_{p,q}=H_p(Y;\Omega_q^{\TopSpin}) \, \Rightarrow \, \Omega_*(\xi^{\Top}_Y),\]
$p,q \geq 0$, respectively.
For $\xi_Y$ orientable these were constructed by Teichner~\cite{teichner-signature}*{Section~II}.
For $\xi_Y$ nonorientable, if there exists a vector bundle $E \to Y$ such that $w_i(E) = w_i \in H^i(Y;\Z/2)$, then the adaptation of the James spectral sequence is constructed in \cite{teichner-star}*{Before~Lemma~2}.

For the general case, it is possible to adapt the construction of \cite{teichner-signature}*{Section~II, pp.~748-9}.
A detailed proof will appear in forthcoming work of Galvin--Teichner--Vesel\'a. We prefer to give credit to that work and not to duplicate effort, so here we will just give an outline of the adaptation.

Here is a description of how adapt the proofs in the literature, that were given for orientable~$\xi_Y$, to the case of nonorientable~$\xi_Y$. We focus on the proof in \cite{teichner-signature}.
On \cite{teichner-signature}*{p.~748}, replace $\BSO$ with $\BO$, $w \colon K(\pi,1) \to K(\Z/2,2)$ with $(w_1,w_2) \colon Y \to K(\Z/2,1) \times K(\Z/2,2)$, and $w_2(\gamma)$ with $(w_1(\gamma),w_2(\gamma))$ in diagram ($\ast$) on p.~748.
In the statement of \cite{teichner-signature}*{Theorem, p.~748}, using stable homotopy theory $\pi_*^{\mathrm{st}}$ for $h_*$, the coefficients are $\pi_q^{\mathrm{st}}(M(\xi|_{\BSpin})) \cong \Omega_q^{\Spin}$.  We replace $\xi|_{\BSpin} \colon \BSpin \to \BSO$ with $\xi|_{\BSpin} \colon \BSpin \to \BO$. The stable homotopy groups of the associated Thom spectrum are still $\Omega_q^{\Spin}$, but now the bordism groups inherit the natural action of $\pi_1(\BO) \cong \Z/2$, where the nontrivial element acts by changing the underlying orientation, and so the coefficients are twisted using this action.
In the proof of \cite{teichner-signature}*{Theorem, p.~748}, the relative Leray--Serre spectral sequence appealed to on \cite{teichner-signature}*{p.~749} makes use of the fact that the fibration is $h$-orientable. However the $h$-orientable assumption is not needed if we use twisted coefficients \cite{whitehead-elements}*{Theorem~XIII.4.9}. Since the fibration is constructed using the pullback along $(w_1,w_2)$, the action of $\pi_1(Y)$ on the coefficients factors through the $\Z/2$ action just described. In this way we obtain the James spectral sequence for $\xi_Y$ nonorientable, without needing to assume the existence of the vector bundle $E \to Y$.
\end{remark}

\begin{assumption}\label{assumption}
	Here and throughout the article, when the coefficients for homology are $\Omega_q^{\Spin}$ or $\Omega_q^{\TopSpin}$, they are assumed to be twisted by $w_1$, in the sense that nontrivial $w_1$ induces multiplication by $-1$.  This will primarily be relevant in this article for $q=4$.
\end{assumption}

One of the virtues of the James spectral sequence is that its $d_2$ differentials can be computed explicitly in low degrees via standard methods in algebraic topology, as follows.  Recall that $\Omega_0^{\Spin} \cong \Z$ and $\Omega_1^{\Spin} = \Omega_2^{\Spin}=\Z/2$; we will use these identifications implicitly in the next proposition.

\begin{proposition}[Teichner]
	\label{prop:differentials}
	Let $Y\xrightarrow{(w_1,w_2)}K(\Z/2,1) \times K(\Z/2,2)$ be a fibration and let $\xi_Y\colon B_Y\to \BO$ be the pullback from \cref{defn:key-pullback}.
		\begin{enumerate}[(i)]
			\item \label{prop:differentials-item-i}
		For $p\leq 4$, the differential $d_2\colon H_p(Y;\Omega_1^{\Spin})\to H_{p-2}(Y;\Omega_2^{\Spin})$ is the dual of the map
        \begin{align*}
		    H^{p-2}(Y;\Z/2) &\to H^p(Y;\Z/2) \\
		x &\mapsto \Sq^2_{w_1,w_2}(x) = \Sq^2(x)+w_1\Sq^1(x)+w_2x.
        \end{align*}
    \item \label{prop:differentials-item-ii}
    For $p\leq 5$, the differential $d_2 \colon H_p(Y;\Omega_0^{\Spin})\to H_{p-2}(Y;\Omega_1^{\Spin})$ is the composition of the reduction of coefficients $H_p(Y;\Z^{w_1})\to H_p(Y;\Z/2)$ with the dual of the map $\Sq^2_{w_1,w_2} \colon H^{p-2}(Y;\Z/2) \to H^p(Y;\Z/2)$ in \eqref{prop:differentials-item-i}.
			\end{enumerate}
\end{proposition}

\begin{proof}
Again, the proof of \cite{teichner-signature}*{Proposition~1,~p.~750} was given in the orientable case. In \cite{teichner-star}*{Section~2} it was explained how to adapt to the nonorientable case, provided there is a vector bundle $E \to Y$ with $w_i(E) = w_i \in H^i(Y;\Z/2)$ for $i=1,2$.
We explain how to perform the necessary adaptations in the general case of $w_1 \neq 0$, without assuming the existence of such an~$E$.

More details on Teichner's proof of \cite{teichner-signature}*{Proposition~1} were given by Orson--Powell in \cite{Orson-Powell}*{Section~5}, so we explain how to perform the adaptation by referencing the Orson--Powell write up. In \cite{Orson-Powell}*{Lemma~5.3}, replace $\operatorname{BSCAT}$ with $\BO$, and $\Sq^2_{w_2(v)}$ with $\Sq^2_{w_1,w_2}$.  To justify the latter replacement, note that on the third line of the displayed equations in the proof of \cite{Orson-Powell}*{Lemma~5.3}, the term $U_n \cup \rho_n^*w_1(v_n) \cup \Sq^1(\rho_n^*(X))$ from the previous line is observed to be zero, because in that paper $w_1(v_n)=0$. In the general case $w_1(v_n) \neq 0$ and that term survives, so we are left with $\Sq^2_{w_1,w_2}$ as asserted.

Further, in the proof of \cite{Orson-Powell}*{Lemma~5.4}, for $0 \leq q \leq 2$, the terms $E^2_{p,q}$ of the James spectral sequence are identified with the $E^2_{p,q}$ terms in an Atiyah--Hirzebruch spectral sequence for a Thom space $H_p(M(v);\pi^{\mathrm{st}}_q)$, where $v \colon B \to \BO$ is a stable vector bundle over $B$  (corresponding to $B_Y$ in our setting).    In the case that $v$ is a nonorientable bundle, the Thom isomorphism theorem identifies this with $H_p(B;(\pi^{\mathrm{st}}_q)^{w_1})$, where the coefficients are now twisted by $w_1$ (recall that after identification of $\pi_q^{\mathrm{st}}$ with $\Omega_q^{\Spin}$, following our convention in \cref{assumption} we cease to  indicate the twisting in the notation). The rest of the proof in \cite{Orson-Powell}*{Lemma~5.4 and Corollary~5.7} proceeds as given in that paper, with straightforward modifications, e.g.\ as already noted, insert $\BO$ in place of $\operatorname{BSCAT}$, and $\Sq^2_{w_1,w_2}$ in place of $\Sq^2_{w_2(v)}$, and then adapt the displayed computation at the end of the proof of \cite{Orson-Powell}*{Corollary~5.7} to include the term $w_1(v) \cup \Sq^1(f^*(u))$ on the second line, and subsequent terms arising from that one.
\end{proof}

The case $p=4$ of \eqref{prop:differentials-item-i} will be used in \cref{section:proof-thm-A}, and the cases $p=4,5$ of \eqref{prop:differentials-item-ii} will be used in \cref{sec:proof-thm-B}.
For our Steenrod square computations in the sequel we will apply the rule that $\Sq^n(x) = x^2$ if $x$ is degree $n$, and the Cartan formula $\Sq^1(x_1x_2)=\Sq^1(x_1)x_2+x_1\Sq^1(x_2)$, without further comment.

\section{Modified surgery applied to stable exotica}\label{sec:applications-modified}

The next proposition explains our strategy for the proof of \cref{thm:main} cf.~ \cite{Kreck84} and \cite{teichner-phd}*{Corollary~5.1.3}.

\begin{proposition}\label{prop:strategy}
	Let $(B,\xi)$ be a normal 1-type for a closed, nonorientable, smooth $4$-manifold with spin universal cover.
	\begin{enumerate}[(i)]
		\item\label{item:i} The $K3$ manifold determines a well-defined element of $\Omega_4(\xi)$.
		\item\label{item:ii} If $[K3]=0\in\Omega_4(\xi)$, then the forgetful map $\Omega_4(\xi) \to \Omega_4(\xi^{\Top})$ is injective. In this case, stable homeomorphism implies stable diffeomorphism for $4$-manifolds with normal 1-type $(B,\xi)$.
		\item\label{item:iii} If $[K3] \neq 0 \in \Omega_4(\xi)$, then $\Omega_4(\xi) \to \Omega_4(\xi^{\Top})$ is not injective.  Let $D$ be a null-bordant $4$-manifold with normal 1-type $(B,\xi)$. Then $D \# K3$ and $D \# 11(S^2 \times S^2)$ are homeomorphic but not stably diffeomorphic.
	\end{enumerate}
\end{proposition}

\begin{remark}\leavevmode
\begin{enumerate}[(i)]
                \item   Note that a null-bordant $4$-manifold $D$ can be directly constructed. Take a 2-complex with fundamental group, and thicken to a $5$-manifold. Thicken the 1-skeleton to realise the desired $w_1$, and thicken the 2-skeleton so as to realise the desired $w_2$. Then take the boundary.
                \item Although it will not be required for the next proof, for definiteness we establish a preferred orientation for $K3$, namely the one for which the signature of the intersection form is $+16$.
              \end{enumerate}
\end{remark}

\begin{proof}
	First we prove \eqref{item:i}.  Let $(\pi,w_1,w_2)$ be the triple determining $(B,\xi)$. Take a model for $B\pi$ such that $(w_1,w_2) \colon B\pi \to K(\Z/2,1) \times K(\Z/2,2)$ is a fibration.  Then by \cref{lemma:def-key-pullback-gives-correct-1-type}, the normal 1-type $\xi \colon B \to \BO$ is fibre homotopy equivalent to the pullback
	\begin{equation}\label{eqn:pullback-defining-B-xi}
		\begin{tikzcd}[column sep = large]
			B_{\pi} \ar[r] \ar[d,"\xi_{\pi}"] & B\pi \ar[d,"{(w_1,w_2)}"] \\ \BO \ar[r,"P_2(-)"] & K(\Z/2,1) \times K(\Z/2,2)
		\end{tikzcd}
	\end{equation}
	of $(w_1,w_2)$ along $\BO\to P_2(\BO)\simeq K(\Z/2,1)\times K(\Z/2,2)$ from \cref{defn:key-pullback}.
We have $\Omega_4^{\Spin}\cong 16 \Z$ and $\Omega_4^{\TopSpin}\cong 8 \Z$ via the signature. Since we assume that $w_1$ is nontrivial, in the smooth and topological James spectral sequences we have
\[E^2_{0,4} \cong H_0(\pi;\Omega_4^{\Spin})\cong \Z\otimes_{\Z\pi}16\Z^{w_1}\cong \Z/2 \text{ and  } E^{2,\Top}_{0,4} \cong H_0(\pi;\Omega_4^{\TopSpin})\cong \Z\otimes_{\Z\pi}8\Z^{w_1}\cong \Z/2.\]
These groups are generated by $K3$ and $E_8$, respectively.

Since $K3$ is simply-connected, every map $K3\to B\pi$ is null homotopic and a choice of $\xi$-structure is the same as the choice of a spin structure.
Since $K3$ has a unique spin structure for each choice of orientation, it remains to show that the choice of orientation does not affect the element. For this we use that $[K3]$ lies in the subgroup $E^{\infty}_{0,4} \subseteq \Omega_4(\xi)$, which, as a quotient of $E^2_{0,4} \cong \Z/2$, has order at most two.  Thus $[K3] = -[K3]$ and so $[K3]$ represents a unique element of $\Omega_4(\xi)$, as claimed.  This proves~\eqref{item:i}.
	
	For the proof of \eqref{item:ii} and \eqref{item:iii} we first show that $[K3]$ generates the kernel of $\Omega_4(\xi)\to\Omega_4(\xi^{\Top})$.
	The map $\Omega_i^{\Spin}\to \Omega_i^{\TopSpin}$ is an isomorphism for $0\leq i\leq 3$, and we saw above that forgetful map $\Omega_4^{\Spin} \to \Omega_4^{\TopSpin}$ is identified, via the signature, with the inclusion $16\Z \to 8\Z$.
	It follows that the map $H_0(\pi;\Omega_4^{\Spin}) \cong \Z/2 \to H_0(\pi;\Omega_4^{\TopSpin}) \cong \Z/2$ is trivial. To see this, let $g \in \pi$ be such that $w_1(g) =-1$. Then
	\[1 \otimes 16 \mapsto 1 \otimes 8 + 1 \otimes 8 = 1 \cdot e \otimes 8 + 1 \cdot g \otimes 8 = 1 \otimes 8 + 1 \otimes (-8) =0.\]
	Thus we see that the kernel of $\Omega_4(\xi)\to\Omega_4(\xi^{\Top})$ is generated by $[K3]\in E^{\infty}_{0,4} \subseteq \Omega_4(\xi)$.
	It follows that $[K3] =0 \in \Omega_4(\xi)$ if and only if the forgetful map is injective.  This shows the first parts of both~\eqref{item:ii} and~\eqref{item:iii}.
	The second part of \eqref{item:ii} now follows directly from \cref{thm:kreck-thm-C-smooth,thm:kreck-thm-C-top}.
	
	For the second part of \eqref{item:iii}, first note that $[D \# K3]$ is bordant to $[K3]$, which by \eqref{item:i} is a uniquely determined element, and thus $[D \# K3]$ represents a nontrivial element of $\Omega_4(\xi)$. On the other hand $[D \# 11(S^2 \times S^2)]=0 \in \Omega_4(\xi)$. The choice of normal 1-smoothings correspond to an action of the group of homotopy automorphisms of $\xi$ on $\Omega_4(\xi)$. It is easy to see that this action preserves the trivial element of $\Omega_4(\xi)$.  Hence there are no choices of normal 1-smoothings for which $[D \# K3]$ and $[D \# 11(S^2 \times S^2)]$ are bordant in $\Omega_4(\xi)$, and hence they are not stably diffeomorphic by \cref{thm:kreck-thm-C-smooth}.
	
	Since $K3$ is homeomorphic to $E_8\#E_8\#3(S^2\times S^2)$ by Freedman's classification of simply-connected $4$-manifolds~\cite{Freedman-82} and $D$ is nonorientable, we have homeomorphisms
	\[D \# K3\cong D\#E_8\#E_8\#3(S^2\times S^2)\cong D\#E_8\#\overline{E}_8\#3(S^2\times S^2)\cong D \# 11(S^2 \times S^2)\]
	as claimed.
\end{proof}

It follows from the proof that whether \eqref{item:ii} or \eqref{item:iii} occurs is governed by whether, for some $k \geq 2$, there is a nontrivial differential $d_k \colon E^k_{k,5-k} \to E^k_{0,4}$. This can only potentially happen for $k \in \{3,4,5\}$ because $\Omega_3^{\Spin}=0$ and the James spectral sequence is first-quadrant. If there is such a nontrivial differential, $[K3]=0 \in \Omega_4(\xi)$ and there are no stably exotic 4-manifolds with normal 1-type $(B,\xi)$. On the other hand if all the differentials $d_k$ with codomain $E^k_{4,0}$ are trivial, then $[K3] \neq 0 \in \Omega_4^(\xi)$ and there is a stably exotic pair of 4-manifolds with normal 1-type $(B,\xi)$.

We can now sketch the proof of \cite{Kreck84}*{Theorem~1}. In Kreck's statement, $w_2$ is not explicitly mentioned, but the statement below can be deduced  from Kreck's proof.
In any case we give the proof of this result because we will use it in \cref{section:proof-thm-A}.

\begin{theorem}[Kreck]
	\label{thm:kreck}
	Let $\pi$ be a group and let $0 \neq w_1\in H^1(\pi;\Z/2)$ and $w_2\in H^2(\pi;\Z/2)$ be given. If $w_1^2=w_2$, then there exist $4$-manifolds $M$ and $M'$ with normal 1-type $(\pi,w_1,w_2)$ that are  homeomorphic but not stably diffeomorphic.
\end{theorem}

\begin{proof}
	Let $\xi_{\pi} \colon B_{\pi} \to \BO$ be the fibration from \cref{defn:key-pullback}. We consider the bordism group $\Omega_4(\xi_{\pi})$ and the associated James spectral sequence.	
	We show that $E^\infty_{0,4}=\Z/2$.
	\begin{claim}
		If $[K3] \neq 0 \in E^{\infty}_{0,4} \subseteq \Omega_4(\xi_{\Z/2})$ then $[K3] \neq 0 \in E^{\infty}_{0,4} \subseteq \Omega_4(\xi_{\pi})$.
	\end{claim}
	To see this claim, note that $w_1$ determines a map $w_1 \colon \pi \to \Z/2$. Take $w_i \neq 0 \in H^i(B\Z/2;\Z/2)$ for $i=1,2$, and define $\xi_{\Z/2} \colon B_{\Z/2} \to \BO$ as in \cref{defn:key-pullback}. Then since $w_1^2=w_2$ for both $\Z/2$ and $\pi$, it follows that $w_1 \colon \pi \to \Z/2$ determines a map of fibrations $(B_{\pi},\xi_{\pi})$ to $(B_{\Z/2},\xi_{\Z/2})$, and hence determines a map of  bordism groups $\Omega_4(\xi_{\pi}) \to \Omega_4(\xi_{\Z/2})$. By naturality of the James spectral sequence there is an associated homomorphism between the $E^{\infty}_{0,4}$ terms. \cref{prop:strategy}~\eqref{item:i} implies that $[K3]$ determines a well-defined class in both bordism classes. In addition we saw in the proof of \cref{prop:strategy} that $[K3]$ generates both (at this point, potentially trivial) $E^{\infty}_{0,4}$ terms, and $\Omega_4(\xi_{\pi}) \to \Omega_4(\xi_{\Z/2})$ sends $[K3] \mapsto [K3]$. Thus if $[K3]$ is nonzero in $E^{\infty}_{0,4} \subseteq \Omega_4(\xi_{\Z/2})$, then certainly $[K3]$ is nonzero in  $E^{\infty}_{0,4} \subseteq \Omega_4(\xi_{\pi})$. This completes the proof of the claim.
	
	Thus we assume that $\pi = \Z/2$, and we show that $[K3] \neq 0 \in \Omega_4(\xi_{\Z/2})$.
	Since $\Omega_0^{\Spin} \cong \Z$, $\Omega_1^{\Spin} = \Omega_2^{\Spin}=\Z/2$, and $\Omega_3^{\Spin}=0$, we have for $E^2_{p,q} \cong H_p(B\Z/2;\Omega_q^{\Spin})$ that
	\[E^2_{0,4}=E^2_{2,2}=E^2_{3,1}=E^2_{4,0}=\Z/2 \text{ and } E^2_{1,3}=0.\]
	On the other hand, $\Omega_4(\xi_{\Z/2})=\Omega_4^{\Pin^+}\cong\Z/16$ by \citelist{\cite{Giambalvo}\cite{Kreck84}*{Proposition~2}} and as we recalled in \cref{prop:pin-4-bordism-groups}.
	Hence for order reasons $E^\infty_{0,4}$   must equal $\Z/2$ (as must $E^\infty_{2,2}$, $E^\infty_{3,1}$, and $E^\infty_{4,0}$).  Since $[K3]$ generates $E^{\infty}_{0,4}$, it follows that $[K3] \neq 0 \in \Omega_4(\xi_{\Z/2})$.
	
	By the claim, $[K3] \neq 0 \in \Omega_4(\xi_{\pi})$. Then \cref{prop:strategy}~\eqref{item:iii} completes the proof of \cref{thm:kreck}.
\end{proof}

We close this section by giving the proof of \cref{thm:C}. 

\begin{proof}[Proof of \cref{thm:C}]
		Let $K$ be a presentation $2$-complex for a presentation of~$G$ realising $\defi(G)$. Then the Euler characteristic of $K$ is $1-\defi(G)$. By Wall \cite{thickenings}*{Proposition~5.1}, there exists a $5$-dimensional thickening $W$ of $K$ with $w_i(\nu W)=w_i$ for $i=1,2$. Let $N:=\partial W$. Then $N$ has normal $1$-type $(G,w_1,w_2)$. Since $W \cup_N W$ is odd-dimensional, $0= \chi(W \cup_N W) = 2\chi(W) -\chi(N)$, and hence $N$ has Euler characteristic $\chi(N) = 2(1-\defi(G))= 2-2\defi(G)$. 
	Let $M:=N\#(S^2\times S^2)$, which thus has Euler characteristic $4-2\defi(G)$. 

To construct $M'$, we will use a manifold $A$ constructed by Akbulut in \cite{akbulut} that is homeomorphic but not stably diffeomorphic to $(S^1\wt\times S^3)\#(S^2\times S^2)$. The bordism group of the normal 1-type of $A$ is $\Z/2$, since $E^2_{p,q} = H_p(S^1;\Omega_q^{\Spin})=0$ for $(p,q) = (1,3)$ and $p>1$.  Since $A$ is not stably diffeomorphic to $(S^1\wt\times S^3)\#(S^2\times S^2)$, which represents the trivial element, $A$ must be stably diffeomorphic to $(S^1\wt\times S^3)\#K3$. 
  Choose an element $g$ of $G$ with $w_1(g)=-1$ and consider~$A$ as an element of $\Omega_4(\xi(G,w_1,w_2))$ by sending a generator of $\pi_1(A)\cong \Z$ to $g$. Perform a 1-surgery on $N\#A$ that identifies $g$ with the generator of $\pi_1(A)$, to obtain the manifold $M'$ whose fundamental group is again $G$.  Since $A$ has Euler characteristic $2$, $N\#A$ has Euler characteristic $2-2\defi(G)$. The surgery changes the Euler characteristic by $\chi(D^2 \times S^2) = 2$, so that $M'$ has Euler characteristic $4-2\defi(G)$ as claimed. This shows \eqref{it:smallex1}.

	To show \eqref{it:smallex2}, start with $N\# (S^1\wt\times S^3)\#(S^2\times S^2)$ and perform a surgery that identifies $g$ with the generator of $\pi_1(S^1)$. This returns $N\#(S^2\times S^2)=M$. Since $A$ and  $(S^1\wt\times S^3)\#(S^2\times S^2)$ are homeomorphic, it follows that $M$ and $M'$ are homeomorphic. Since $A$ is stably diffeomorphic to $(S^1\wt\times S^3)\#K3$, it similarly follows that $M'$ is stably diffeomorphic to $N \# K3$.
	By \cref{prop:strategy}~\eqref{item:iii}, the manifolds $N\#K3$ and $N\#11(S^2\times S^2)$ are not stably diffeomorphic and hence $M$ and $M'$ are not stably diffeomorphic. This completes the proof of \eqref{it:smallex2}.
	
	Items \eqref{it:smallex4}--\eqref{it:smallex5} follow from similar properties of $A$; that is, $A$ is obtained from $(S^1\wt\times S^3)\#(S^2\times S^2)$ by a Gluck twist \cite{akbulut} and the orientation double covers of $A$ and $(S^1\wt\times S^3)\#(S^2\times S^2)$ are diffeomorphic \cite{Torres-JKTR}. See \cite{torres}*{Section~3.4} for a more detailed argument.
\end{proof}

\section{Stable homeomorphism implies stable diffeomorphism when \texorpdfstring{$w_1^3\neq w_1w_2$}{w1 cubed is not equal to w1w2}}\label{section:proof-thm-A}
For this section, we fix a normal $1$-type $(\pi,w_1,w_2)$. Note that \cref{assumption} remains in force, so bordism group coefficients are always twisted using~$w_1$.

We will deduce \cref{thm:main} from the following result, which we will prove at the send of this section.
\begin{theorem}
	\label{thm:d3-diff}
	The composition
	\[H_3(\pi;\Omega_2^{\Spin})\twoheadrightarrow H_3(\pi;\Omega_2^{\Spin})/\im d_2\xrightarrow{d_3^{\pi}} H_0(\pi;\Omega_4^{\Spin})\cong \Z/2\] is dual to the map \[\Hom(H_0(\pi;\Omega_4^{\Spin}),\Z/2) \to \Hom(H_3(\pi;\Omega_2^{\Spin}),\Z/2) \cong H^3(\pi;\Z/2)\] given by $1\mapsto w_1^3+w_1w_2\in H^3(\pi;\Z/2)$. In particular, $d_3^\pi$ is trivial if and only if $w_1^3=w_1w_2$.
\end{theorem}

For better readability, we will often just say that $d_3^\pi$ is dual to $w_1^3+w_1w_2$.

\begin{proof}[Proof of \cref{thm:main} assuming \cref{thm:d3-diff}.]
	Since \eqref{it:main-2} was already proved in \cref{thm:kreck}, we only show \eqref{it:main-1} and \eqref{it:main-3}. If $w_1^3\neq w_1w_2$, then $d_3^{\pi}$ is nontrivial by \cref{thm:d3-diff}. Hence $[K3]=0\in \Omega_4(\xi_{\pi})$ and thus no stably exotic $4$-manifolds with normal $1$-type $(\pi,w_1,w_2)$ exist by \cref{prop:strategy}~\eqref{item:ii}. This implies~\eqref{it:main-1}.
	
	Now assume that $w_1^3=w_1w_2$ and that $\pi$ has cohomological dimension at most $3$. The differential $d_3^{\pi}$ is trivial by \cref{thm:d3-diff}. Since $\pi$ has cohomological dimension at most $3$, all higher differentials are trivial as well. It follows that $[K3]\neq 0\in \Omega_4(\xi_{\pi})$. The existence of stably exotic $4$-manifolds with normal $1$-type $(\pi,w_1,w_2)$ now follows from \cref{prop:strategy}~\eqref{item:iii}. This completes the proof of \eqref{it:main-3}.
\end{proof}
	
To investigate $d_3^{\pi}$, first we  consider the universal example
   \[X:=K_1\times K_2:=K(\Z/2,1)\times K(\Z/2,2)\]
    determining a fibration $\xi_X \colon B_X \to \BO$ as in \cref{defn:key-pullback}, using $v_j:=p_j^*x_j$ for $j=1,2$, where $p_j\colon X\to K_j$ is the projection and $x_j\in H^j(K_j;\Z/2)$ is the generator.
	
	\begin{lemma}
		\label{lem:nontrivial-X-diff}
		The differential \[d_3^X \colon H_3(X;\Omega_2^{\Spin})/\im d_2^{X} \cong H_3(X;\Z/2)\to H_0(X;\Omega_4^{\Spin})/\im d_2^{X} \cong \Z/2\] in the James spectral sequence for $\xi_X \colon B_X \to \BO$ is nontrivial.
	\end{lemma}

\begin{proof}		
	We consider the map $\Theta\colon \RP^\infty\to X$, with $u_1:=\Theta^*v_1\neq 0$ the generator of $H^1(\RP^\infty;\Z/2)$ and $u_2:=\Theta^*v_2=0 \in H^2(\RP^\infty;\Z/2)$. Since $X= K_1 \times K_2$ is a product of Eilenberg-Maclane spaces, this determines $\Theta$ up to homotopy.  We consider the fibration $\xi_{\RP^\infty} \colon B_{\RP^\infty} \to \BO$ determined as in \cref{defn:key-pullback} by $\RP^\infty$, $u_1$, and $u_2$. We have an induced map $\Theta_* \colon \Omega_4(\xi_{\RP^\infty}) \to \Omega_4(\xi_X)$.
	
	A 4-manifold or a 5-dimensional bordism $W$ admits the structure of a normal bordism over $\xi_{\RP^\infty}$ if and only if $w_2(\nu_W) =0$, which holds if and only if  $\nu_W$ admits a $\Pin^+$ structure, which in turn holds if and only if $TW$ admits a $\Pin^-$ structure by \cref{prop:pin-structures}.
	We deduce that $\Omega_4(\xi_{\RP^\infty})\cong\Omega_4^{\Pin^-}$, and then we recall that $\Omega_4^{\Pin^-}$ vanishes by \cref{prop:pin-4-bordism-groups}.
	In particular, $[K3]$ is trivial in $\Omega_4(\xi_{\RP^\infty}) \cong \Omega_4^{\Pin^-}=0$.

	We now deduce that the differential $d_3^{\RP^{\infty}} \colon H_3(\RP^\infty;\Z/2)/\im d_2^{\RP^{\infty}} \to H_0(\RP^\infty;\Z^{u_1}) \cong \Z/2$ is nontrivial. To do so, we show that $E^2_{2,3}=E^3_{4,1}=E^2_{5,0}=0$ in the James spectral sequence computing $\Omega_*(\xi_{\RP^\infty})$, and hence only $d_3^{\RP^\infty}$ can possibly kill the $E^*_{0,4}$ term; since $[K3]=0$ this term must be killed by one of $d_3^{\RP^{\infty}}$, $d_4^{\RP^{\infty}}$, or $d_5^{\RP^{\infty}}$.
	
	First, $E^2_{2,3} = H_2(\RP^\infty;\Omega_3^{\Spin})=0$ because $\Omega_3^{\Spin} =0$. Next, since $u_2=0$, we have
	\[(\Sq^2+u_1\Sq^1)(u_1^2)=\Sq^2(u_1^2) + u_1 (\Sq^1u_1) u_1 + u_1^2(\Sq^1u_1) =u_1^4.\]
	Hence using \cref{prop:differentials}~\eqref{prop:differentials-item-i} we can compute $d_2^{\RP^{\infty}} \colon E_{4,1}^2 \to E^2_{2,2}$. We have $E_{4,1}^2 \cong H_4(\RP^\infty;\Omega_1^{\Spin}) \cong \Z/2$, with generator dual to $u_1^4$, and $E_{2,2}^2 \cong H_2(\RP^\infty;\Omega_2^{\Spin}) \cong \Z/2$, with generator dual to $u_1^2$. Thus by \cref{prop:differentials}~\eqref{prop:differentials-item-i},  $d_2^{\RP^{\infty}}$ is an isomorphism and so $E^3_{4,1}=0$. We also have $E^2_{5,0}=H_5(\RP^\infty;\Z^-)=0$.  Thus we have shown that $E^2_{2,3}=E^3_{4,1}=E^2_{5,0}=0$, as asserted. It follows that the only possibly nontrivial differential that can hit $E^k_{0,4}$, for some $k \geq 2$, is the differential $d_3^{\RP^{\infty}} \colon E^3_{3,2}\to E^3_{0,4}=\Z/2$. Since we showed that $K3$ is null-bordant in $\Omega_4(\xi_{\RP^{\infty}})$, this differential has to be nontrivial. By naturality of the James spectral sequence, the square
	\[\begin{tikzcd}
		H_3(\RP^\infty;\Z/2)/\im d_2^{\RP^{\infty}}\ar[d,"d_3^{\RP^{\infty}}","\cong"']\ar[r,"\Theta_*"']&H_3(X;\Z/2)/\im d_2^{X}\ar[d,"d_3^X"]\\
		H_0(\RP^\infty;\Z^{u_1})\ar[r,"\cong",,"\Theta_*"'] & H_0(X;\Z^{v_1})
	\end{tikzcd}\]
	commutes. Since $d_3^{\RP^{\infty}}$ is nontrivial (in fact it is an isomorphism because both domain and codomain are order two), and the bottom horizontal map is an isomorphism (again between groups of order two),  it follows that $d_3^X$ is nontrivial, as required.
\end{proof}

	We can now determine $d_3^X$, as follows.

	\begin{lemma}
		\label{lem:d3-X}
		The composition $H_3(X;\Z/2)\twoheadrightarrow H_3(X;\Z/2)/\im d_2^{X}\xrightarrow{d_3^X} H_0(X;\Z^{v_1}) \cong \Z/2$ is dual to $1 \mapsto v_1^3+v_1v_2$.
	\end{lemma}

\begin{proof}
	By the K\"{u}nneth theorem
	\[H_3(X;\Z/2)\cong H_3(K_2;\Z/2)\oplus \big(H_2(K_2;\Z/2)\otimes H_1(K_1;\Z/2)\big) \oplus H_3(K_1;\Z/2)\cong H_3(K_2;\Z/2)\oplus(\Z/2)^2.\]
	By naturality of the James spectral sequences, comparing those for $(B_{K_2},\xi_{K_2})$ and $(B_X,\xi_X)$ via the inclusion of the second factor $\{*\}\times K_2 \hookrightarrow X$, there is a commutative diagram
	\[\begin{tikzcd}
		H_3(K_2;\Z/2)/\im d_2^{K_2}\ar[d,"d_3^{K_2}"]\ar[r]&H_3(X;\Z/2)/\im d_2^{X}\ar[d,"d_3^{X}"]\\
		H_0(K_2;\Z)\ar[r]&H_0(X;\Z^{v_1}).
	\end{tikzcd}\]
	Since $H_0(K_2;\Z)\cong \Z$, and $H_3(K_2;\Z/2)$ is 2-torsion, $d_3^{K_2}$ is trivial and hence $d_3^X$ is trivial on the image of $H_3(K_2;\Z/2)$. The quotient $H_3(X;\Z/2)/H_3(K_2;\Z/2)\cong (\Z/2)^2$ is generated by the two classes dual to $v_1v_2$ and $v_1^3$.
	
	Consider the map $\Theta'\colon \RP^\infty\to X$ with $(\Theta')^*v_1\neq 0$ and $(\Theta')^*(v_2) = (\Theta')^*(v_1^2)$, i.e.\ $(\Theta')^*(v_2+v_1^2)=0$.
	Again since $X$ is a product of Eilenberg-Maclane spaces, this determines $\Theta'$ up to homotopy.
	Write $u'_1:=(\Theta')^*v_1 \in H^1(\RP^\infty;\Z/2)$ and $u'_2:=(\Theta')^*v_2 \in H^2(\RP^\infty;\Z/2)$.  This gives rise to a fibration $\xi'_{\RP^\infty} \colon B'_{\RP^\infty} \to \BO$ as in \cref{defn:key-pullback}, with a corresponding bordism group $\Omega_4(\xi'_{\RP^\infty})$ and a map $\Omega_4(\xi'_{\RP^\infty}) \to \Omega_4(\xi_X)$. Since $u_2'=(u_1')^2$, it follows from the proof of \cref{thm:kreck} that $[K3]$ is nontrivial in $\Omega_4(\xi'_{\RP^\infty})$. In particular the differential $d_3^{\RP^{\infty}} \colon H_3(\RP^{\infty};\Z/2)/\im d_2^{\RP^{\infty}}  \to H_0(\RP^{\infty};\Z^{u_1'})$ must be trivial. From the square
	\[\begin{tikzcd}
		H_3(\RP^\infty;\Z/2)/\im d_2^{\RP^{\infty}}\ar[d,"d_3^{\RP^{\infty}}","0"']\ar[r,"\Theta_*'"'] & H_3(X;\Z/2)/\im d_2^{X}\ar[d,"d_3^X"]\\
		H_0(\RP^\infty;\Z^{u_1'})\ar[r,"\cong","\Theta_*'"'] & H_0(X;\Z^{v_1}),
	\end{tikzcd}\]
	we deduce that the right-down composition is trivial. Taking $\Hom_{\Z/2}(-,\Z/2)$ duals, and using the map $\ev \colon H^3(X;\Z/2) \xrightarrow{\cong} H_3(X;\Z/2)^*$ from universal coefficients, we see that the composition
	\[\Z/2 \cong H_0(X;\Z^{v_1})^* \xrightarrow{\ev^{-1} \circ (d_3^X)^*} H^3(X;\Z/2) \xrightarrow{(\Theta')^*} H^3(\RP^\infty;\Z/2) \]
	is trivial.
	Since $(\Theta')^*(v_1^3) = (u'_1)^3$ and $(\Theta')^*(v_1v_2) = u_1'u_2' = (u'_1)^3$ are nontrivial in $H^3(\RP^{\infty};\Z/2)$, the differential $d_3^X$ is neither dual to $1 \mapsto v_1^3$ nor to $1 \mapsto v_1v_2$. Since $d_3^X$ is nontrivial by \cref{lem:nontrivial-X-diff}, there is one remaining option, and $d_3^X$ must be dual to $1 \mapsto v_1^3+v_1v_2$.
\end{proof}

\begin{proof}[Proof of \cref{thm:d3-diff}]
	One more application of naturality yields the square:
	\[\begin{tikzcd}[column sep = large]
		H_3(\pi;\Z/2)/\im d_2^\pi\ar[d,"d_3^{\pi}"]\ar[r,"{(w_1,w_2)_*}"] & H_3(X;\Z/2)/\im d_2^X\ar[d,"d_3^X"]\\
		H_0(\pi;\Z^{w_1})\ar[r,"\cong"] & H_0(X;\Z^{v_1}).
	\end{tikzcd}\]
	Since $d_3^X$ is dual to $v_1^3+v_1v_2$ by \cref{lem:d3-X}, and the map $B\pi \to X = K(\Z/2,1) \times K(\Z/2,2)$ pulls the universal classes $v_1 \in H^1(X;\Z/2)$ and $v_2 \in H^2(X;\Z/2)$ back to $w_1$ and $w_2$ respectively, it follows that $d_3^\pi$ is dual to $1 \mapsto w_1^3 + w_1w_2$.
\end{proof}

\section{The universal space \texorpdfstring{for $w_1^3 + w_1w_2 =0$}{when w1 cubed equals w1w2}}
\label{sec:proof-thm-B}

Our next goal is to understand the class $[K3]$ in the case that $w_1^3 + w_1w_2 =0$ and prove \cref{thm:B}. For this we consider the fibration sequence
\begin{equation}\label{eqn:defn-of-F}
    F \xrightarrow{u} X \xrightarrow{\iota_1^3+\iota_1\iota_2} K(\Z/2,3),
\end{equation}
where as above
\[
    X :=  K_1 \times K_2 := K(\Z/2,1)\times K(\Z/2,2),
    \]
and, for $i=1,2$,
\[
\iota_i  \in H^i(X;\Z/2)
\]
is the pullback of the generator of $H^i(K_i;\Z/2) \cong \Z/2$ under the projection $p_i \colon X \to K_i$.
By definition $F$ is  the homotopy fibre of the map  $\iota_1^3+\iota_1\iota_2$.

We will consider the fibration $\xi_F \colon B_F \to \BO$, constructed
with \begin{equation}\label{defn:v-i}
  v_i:=u^*\iota_i \in H^i(F;\Z/2),\,\,\, i=1,2,
  \end{equation}
as in \cref{defn:key-pullback}.
Note that $v_1^3 + v_1v_2 = u^*(\iota_1^3 + \iota_1\iota_2) =0$, as this corresponds to consecutive maps in the fibration sequence~\eqref{eqn:defn-of-F}.
In addition, if we have normal 1-type data $(\pi,w_1,w_2)$ with $w_1^3 + w_1w_2 =0$, the composition $B\pi \xrightarrow{(w_1,w_2)} K_1 \times K_2 \xrightarrow{\iota_1^3+\iota_1\iota_2} K(\Z/2,3)$ is null-homotopic, and so there is a lift $f\colon B\pi \to F$, resulting in a map $B_\pi \to B_F$ over $\BO$ and a corresponding map $\Omega_4(\xi_\pi) \to \Omega_4(\xi_F)$.
We will prove the following statement.

\begin{theorem}\label{thm:vanishing-K3-in-Omega-4-F}
The class $[K3]$ lies in the kernel of $\Omega_4(\xi_\pi) \to \Omega_4(\xi_F)$.
\end{theorem}

We will use this statement to compute the  $d_4$ differential in the spectral sequence
\[H_p(F;\Omega_q^{\spin}) \,\Rightarrow\, \Omega_4(\xi_F).\]
Then, by comparing spectral sequences, we will use this to give criteria on $(\pi,w_1,w_2)$ with implications on the fate of $[K3]$ in $\Omega_4(\xi_\pi)$, and hence on the existence of stable exotica with this normal 1-type.

To begin, we determine the homotopy type of $F$.

\begin{proposition}
	\label{prop:hom-groups-F}
	The space $F$ is 3-coconnected with $\pi_1(F)\cong \Z/2$, $\pi_2(F)\cong (\Z/2)[\Z/2]$, and trivial $k$-invariant. Furthermore, $v_1\neq 0$ and $v_2\neq v_1^2$.
\end{proposition}

\begin{proof}
	The long exact sequence of homotopy groups for the fibration sequence~\eqref{eqn:defn-of-F} defining $F$ yields
    \[0 \to \pi_1(F) \xrightarrow{u_*} \pi_1(X) \to 0,\]
    which shows that $u$ induces an isomorphism $\pi_1(F)\cong \pi_1(X)\cong \Z/2$.

    Thus $v_1:=u^*\iota_1\in H^1(F;\Z/2)$ is nontrivial.
    Since there is a space $Y$ with cohomology classes $a_i\in H^i(Y;\Z/2)$ for $i=1,2$, such that $a_1^3=a_1a_2$ and $a_2\neq a_1^2$ (see \cref{remark-showing-ii-does-not-imply-iii} for an example) and $F$ is the universal space with cohomology classes $v_1,v_2$ such that $v_1^3=v_1v_2$, we must have that $v_2\neq v_1^2$.

    The long exact sequence of homotopy groups for the fibration sequence defining $F$ also shows that there is a short exact sequence
	\begin{equation}\label{eqn:ses-for-pi2-F}
	    0\to \pi_3(K(\Z/2,3)) \cong \Z/2\to \pi_2(F)\to \pi_2(X) \cong \Z/2\to 0.
	\end{equation}
	Let $i_2\colon K_2\to X$ be the factor inclusion. Then $i_2^*(\iota_1)=0$, so  $i_2^*(\iota_1^3+\iota_1\iota_2)=0$, and hence $i_2$ admits a lift to $F$. This implies that the  sequence~\eqref{eqn:ses-for-pi2-F} splits, and hence \[\pi_2(F)\cong \Z/2\oplus\Z/2\] as an abelian group. Similarly, let $s\colon K_1\to X$ be determined by $s^*\iota_1=x$ and $s^*\iota_2=x^2$ where, $x$ is the generator of $H^1(K_1;\Z/2)$. Then $s^*(\iota_1^3+\iota_1\iota_2)=x^3+x^3=0$ and hence $s$ admits a lift to $F$. This gives a splitting of the 2-connected map $c\colon F\to K_1$ and shows that $F$ has trivial $k$-invariant.
	
	It remains to determine the $\Z[\pi_1(F)]$-module structure of $\pi_2(F)$. The only possibilities are the trivial action and $\pi_2(F) \cong (\Z/2)[\Z/2]$.
    Suppose for a contradiction that the $\pi_1$-action on $\pi_2$ is trivial. Then we would have $F\simeq K_1 \times K_2\times K_2$.
    In the $\Z/2$-cohomology ring of $K_1 \times K_2\times K_2$, multiplication with the unique nontrivial class in degree 1 is injective, by the K\"{u}nneth theorem and because $H^*(K_1;\Z/2) \cong (\Z/2)[x]$. Hence $v_1^2 \neq v_2$ implies $v_1^3\neq v_1v_2$,  contradicting that $v_1^3= v_1v_2$ by construction.
    We see that $\pi_2(F)\cong (\Z/2)[\Z/2]$ as claimed.
\end{proof}

\begin{corollary}
	\label{cor:wtF}
    Let $\wt F$ be the universal cover of $F$.
	There is a $\Z/2$-equivariant homotopy equivalence $\wt F\simeq K_2\times K_2 \times S^{\infty}$, where $\Z/2$ acts via the deck transformations on $\wt F$ and via $(x,y,z)\mapsto (y,x,-z)$ on $K_2\times K_2 \times S^{\infty}$.
\end{corollary}

\begin{proof}
    The space $K_2\times K_2 \times S^{\infty}$ has the same homotopy groups as $\wt{F}$. Moreover the fixed point sets of the $\Z/2$ action on both $\wt{F}$ and $K_2\times K_2 \times S^{\infty}$ are empty. Hence the spaces are equivariantly homotopy equivalent by the equivariant Whitehead theorem; see e.g.~\cite{matumoto}*{Theorem~5.3}.
\end{proof}

Having determined the homotopy type of $F$, we can now give the following description of the classes $u^*\iota_i \in H^i(F;\Z/2)$ for $i=1,2$.

\begin{lemma}
	\label{lem:cohom-classes-F}
	We have $H^1(F;\Z/2)\cong \Z/2$ and $H^2(F;\Z/2)\cong (\Z/2)^2$. The class $v_1:=u^*\iota_1\in H^1(F;\Z/2)$ is nontrivial, and the class $v_2:=u^*\iota_2\in H^2(F;\Z/2)$ is the unique class such that $v_2\neq v_1^2$ but $s^*v_2=s^*v_1^2$ for every section $s\colon K_1\to F$ of the 2-connected map $c \colon F \to K_1$.
\end{lemma}

\begin{proof}
	By \cref{prop:hom-groups-F}, $\pi_1(F)\cong \Z/2$ and thus $H^1(F;\Z/2)\cong \Z/2$. Again by \cref{prop:hom-groups-F}, $v_1\neq0$ and $v_2\neq v_1^2$.

    Looking at the Leray--Serre spectral sequence for $\wt{F} \to F \xrightarrow{c} K_1$ with $E^2_{p,q} \cong H_p(K_1;H_q(\wt F;\Z/2))$ and  converging to $H_{p+q}(F;\Z/2)$, analysing the $p+q=2$ anti-diagonal yields an exact sequence
    \[H_0(K_1;H_2(\wt F;\Z/2))\to H_2(F;\Z/2)\to H_2(K_1;\Z/2)\to 0.\]
    By \cref{cor:wtF},
	\[H_0(K_1;H_2(\wt F;\Z/2))\cong \Z/2\otimes_{(\Z/2)[\Z/2]}(\Z/2)[\Z/2]\cong \Z/2.\]
	Since the $k$-invariant of $F$ is trivial, there exists a section $s\colon K_1\to F$ and thus there are no differentials into $H_0(K_1;H_2(\wt F;\Z/2))$. Hence we have a short exact sequence $0 \to \Z/2 \to H_2(F;\Z/2) \to \Z/2 \to 0$ and so \[H^2(F;\Z/2)\cong H_2(F;\Z/2)\cong (\Z/2)^2,\] as claimed.

 Since $s\colon K_1\to F$  induces an isomorphism on fundamental groups, it induces an isomorphism on $H^1(-;\Z/2)$, and hence $s^*(v_1) \neq 0 \in H^1(K_1;\Z/2)$.

We saw that $v_2 \neq v_1^2$ in the proof of \cref{prop:hom-groups-F}, so it remains to show that $s^*v_2=s^*v_1^2 \in H^2(K_1;\Z/2)$ for every section $s$. So, fix a section $s \colon K_1\to F$. Since $v_1^3 = v_1v_2 \in H^3(F;\Z/2)$  we have $s^*(v_1^3)=s^*(v_1v_2)$, and therefore
\[
s^*(v_1)\big(s^*(v_1^2) + s^*(v_2)\big)=0 \in H^3(K_1;\Z/2).
\]
 Since $s^*v_1\neq 0 \in H^1(K_1;\Z/2)$ and $H^*(K_1;\Z/2)$ is a polynomial ring in a single variable ($K_1 \simeq \RP^\infty$), this implies that $s^*v_1^2 + s^*v_2=0$, and so $s^*v_1^2 = s^*v_2 \in H^2(K_1;\Z/2)$, as claimed.

 The uniqueness of $v_2$ satisfying this follows from the fact that $H^2(F;\Z/2)$ only has four elements: the element in one summand of $H^2(F;\Z/2) \cong \Z/2 \oplus \Z/2$ is determined by $s^*v_2 = s^*v_1^2$. This leaves two possibilities remaining, but one of them is $v_1^2$, and so the fact that $v_2 \neq v_1^2$ leaves a single element.
\end{proof}

As promised, we will show that $0=[K3]\in\Omega_4(\xi_F)$, where $\xi_F\colon B_F\to \BO$ is the fibration determined by $(F,v_1,v_2)$ as in \cref{defn:key-pullback}.
For this we will first compute the bordism group over $\wt{F}$, where $p\colon \wt{F}\to F$ is the universal cover. To do so, we first calculate $p^*v_2$.

\begin{lemma}
	\label{cor:cohom-wtF}
	Let $p\colon \wt F\to F$ be the universal covering. Under the homotopy equivalence from \cref{cor:wtF}, \[p^*v_2=(\iota_2,\iota_2)\in H^2(K_2;\Z/2) \oplus H^2(K_2;\Z/2)\cong H^2(K_2 \times K_2\times S^\infty;\Z/2)\cong H^2(\wt F;\Z/2).\]
\end{lemma}

\begin{proof}
	Consider the exact sequence
	\begin{equation}\label{eqn:leray-serre-cohomology-for-F}
	H^2(K_1;\Z/2)\xrightarrow{c^*} H^2(F;\Z/2)\xrightarrow{p^*} H^0(K_1;H^2(\wt F;\Z/2)),
    \end{equation}
    which again  follows from the Leray--Serre spectral sequence for $\wt{F} \xrightarrow{p} F \xrightarrow{c} K_1$, this time the cohomology version.  Note that \[H^0\big(K_1;H^2(\wt F;\Z/2)\big) \cong H^0\big(\Z/2;H^2(K_2;\Z/2\big) \oplus H^2\big(K_2;\Z/2)\big) \]
    is isomorphic to the fixed subgroup of the coefficients $H^2(K_2;\Z/2) \oplus H^2(K_2;\Z/2)$ under the $\Z/2$-action $(x,y) \mapsto (y,x)$ from \cref{cor:wtF}.

	Since neither $v_2=0$ nor $v_2=v_1^2$ by \cref{lem:cohom-classes-F}, $v_2$ is not pulled back from $K_1$ under $c$.
    Hence by exactness of~\eqref{eqn:leray-serre-cohomology-for-F}, it follows that $p^*v_2\neq 0$. We must therefore have \[p^*v_2=(\iota_2,\iota_2)\in H^2(K_2;\Z/2) \oplus H^2(K_2;\Z/2),\] because this is the unique nontrivial fixed point of the $\Z/2$-action.
\end{proof}

We can now compute the bordism group $\Omega_4(\xi_{\wt{F}})$, where $\xi_{\wt F}\colon B_{\wt F}\to \BO$ is the fibration determined by $(\wt F,0,p^*v_2)$ as in \cref{defn:key-pullback}.
For this we will need to know the homology of $K_2$ in low degrees.

\begin{proposition}\leavevmode
\label{prop:homology-K2-low-degrees}
\begin{enumerate}
    \item\label{item:prop-K2-low-degrees-1}
    The homology $H^*(K_2;\Z/2)$ is a polynomial algebra generated by $\Sq^I(\iota_2)$, where $I$ ranges over the admissible sequences $(i_1,i_2,\dots,i_k)$ of excess less than two, and $\Sq^I := \Sq^{i_1} \circ \cdots \Sq^{i_k}$.
    \item\label{item:prop-K2-low-degrees-2} $H_1(K_2;\Z/2)=0$, $H_2(K_2;\Z/2) \cong H_3(K_2;\Z/2) \cong H_4(K_2;\Z/2) \cong \Z/2$, and $H_5(K_2;\Z/2) \cong \Z/2 \oplus \Z/2$.
\item\label{item:prop-K2-low-degrees-3} $H_1(K_2;\Z)=0$, $H_2(K_2;\Z) \cong \Z/2$, $H_3(K_2;\Z)=0$, and $H_4(K_2;\Z) \cong \Z/4$.
\end{enumerate}
\end{proposition}

\begin{proof}
For the statement of \eqref{item:prop-K2-low-degrees-1} we refer to \cite{Hatcher-alg-top}*{pp.~499-500} or \cite{Mosher-Tangora}*{p.~27}. Admissible means that $i_j\geq 2i_{j+1}$ for each $j$, the degree $d(I)$ is $\sum i_j$, and the excess is $2i_1 -d(I)$.

For \eqref{item:prop-K2-low-degrees-2}, the first and second homology groups are straightforward. For degrees higher than two we  deduce the homology from the cohomology given in \eqref{item:prop-K2-low-degrees-1}. The only admissible sequence of degree one yields $\Sq^1(\iota_2)$ and so $H^3(K_2;\Z/2) \cong \Z/2$. There are no degree two admissible sequences of excess less than two, and so $H^4(K_2;\Z/2) \cong \Z/2$ generated by $\iota_2^2$.  Finally, from degree three admissible sequences of excess less than two, we obtain $\Sq^2\Sq^1(\iota_2)$, and so $H^5(K_2;\Z/2) \cong (\Z/2)^2$  generated by $\{\iota_2\Sq^1\iota_2, \Sq^2\Sq^1(\iota_2)\}$.

For the integral homology in \eqref{item:prop-K2-low-degrees-3}, again the first and second homology groups are straightforward, and for the third and fourth homology we refer to \cite{clement}*{Appendix~C}.
\end{proof}

\begin{lemma}
	\label{lem:wtF-bordism}
	Let $Y_1$ and $Y_2$ be two copies of $K_2$, i.e.\ $Y_i:=K_2$ for $i=1,2$. Let $x_i$ be the generator of $H^2(Y_i;\Z/2)$, and let $Y:=Y_1\times Y_2\times S^\infty$.  For $i=1,2$ let $\xi_{Y_i} \colon B_{Y_i} \to \BO$ be the fibration determined by \cref{defn:key-pullback} and the data $(Y_i,0,x_i)$. Let $\xi_{Y} \colon B_Y \to \BO$ be the fibration determined by $(Y,0,x_1+x_2)$.
	\begin{enumerate}[(a)]
		\item\label{it:wtF-i} $\Omega_4(\xi_{Y_i})\cong\Omega_4^{SO}\cong \Z$.
		\item\label{it:wtF-ii} $\Omega_4(\xi_{Y})\cong (\Z\oplus\Z)/(4,-4)$, where the two summands are the images of $\Omega_4(\xi_{Y_i})\cong \Z$ for $i=1,2$.
		\item\label{it:wtF-iii} The homotopy equivalence $\wt{F}\simeq Y$ from \cref{cor:wtF} induces an isomorphism $\Omega_4(\xi_{\wt{F}})\cong \Omega_4(\xi_{Y})$.
        \item\label{it:wtF-iv} The image of $\Omega_4(\xi_{\wt{F}^{(2)}})$ in $\Omega_4(\xi_{\wt{F}})$ is infinite cyclic subgroup $(8\Z\oplus 8\Z)/(8,-8)$ in $(\Z\oplus\Z)/(4,-4)$, where $\wt{F}^{(2)}$ is the $2$-skeleton of $\wt F$ in any CW-structure of $F$.
        \item\label{it:wtF-v} The differential $d_3\colon E_{5,0}^3\to E_{2,2}^3$ in the James spectral sequence converging to $\Omega_4(\xi_Y)$ is trivial.
	\end{enumerate}
\end{lemma}

\begin{proof}
	To prove part~\eqref{it:wtF-i}, we consider the pullback square in \cref{defn:key-pullback}, and augment it with another pullback square:
    \[\begin{tikzcd}[row sep = small]
        B_{Y_i} \ar[r] \ar[d] & K_2 \ar[d] \ar[r] & * \ar[d] \\
        \BO \ar[r,"w_1 \times w_2"] & K_1 \times K_2 \ar[r] & K_1.
    \end{tikzcd}\]
The concatenation of two pullback squares is again a pullback. Hence $B_{Y_i}$ is the pullback of the outer rectangle, which we recognise as $\BSO$. So $\Omega_4(\xi_{Y_i})\cong\Omega_4^{SO}\cong \Z$ as desired.  This proves \eqref{it:wtF-i}, but for later reference we show the entries $E^2_{p,q} \cong H_p(Y_i;\Omega_q^{\Spin})$ with $p+q=4$ in the James spectral sequence (\cref{remark:James-SS}) for $\Omega_4(\xi_{Y_i})$. The coefficients are untwisted due to the $0$ in $(Y_i,0,x_i)$, and for the computations of the homology groups we apply \cref{prop:homology-K2-low-degrees}.

\begin{center}
\begin{tikzpicture}[scale=0.8]

\draw[step=2.0,black,xshift=1cm,yshift=1cm] (0,0) grid (10.5, 5.5);

\draw (1,2) -- (11.5,2);
\draw (1,4) -- (11.5,4);
\draw (1,6) -- (11.5,6);

\node at (0.5,1.5) {$0$};
\node at (0.5,2.5) {$1$};
\node at (0.5,3.5) {$2$};
\node at (0.5,4.5) {$3$};
\node at (0.5,5.5) {$4$};
\node at (0.5,6.5) {$q$};

\node at (2,0.5) {$0$};
\node at (4,0.5) {$1$};
\node at (6,0.5) {$2$};
\node at (8,0.5) {$3$};
\node at (10,0.5) {$4$};
\node at (12,0.5) {$p$};

\node at (10,1.5) {$\Z/4$};

\node at (8,2.5) {$\Z/2$};

\node at (6,3.5) {$\Z/2$};

\node at (4,4.5) {$0$};

\node at (2,5.5) {$16\Z$};



\end{tikzpicture}
\end{center}

The term $E^2_{0,4} \cong H_0(K_2;\Omega_4^{\Spin})$ (with untwisted coefficients) is isomorphic to $16\Z$ via the signature. Since $\Omega_4^{\SO} \cong \Z$, generated by $\CP^2$, and again detected by the signature, we must have that all terms on the $E^2$ page with $p+q=4$ survive to the $E^\infty$ page and give the iterated graded groups of the filtration $16\Z \leq 8\Z \leq 4\Z \leq \Z$.

Next we show \eqref{it:wtF-ii}. For this we consider the James spectral sequence for the computation of $\Omega_4(\xi_Y)$.
	First note that by the Künneth theorem  for $k\leq 3$ we have $H_k(Y;\Z/2)\cong H_k(Y_1;\Z/2)\oplus H_k(Y_2;\Z/2)$, which vanishes for $k=1$ and is $\Z/2 \oplus \Z/2$ for $k=2,3$ by \cref{prop:homology-K2-low-degrees}. Furthermore again by the Künneth theorem  and \cref{prop:homology-K2-low-degrees}, \[H_4(Y;\Z)\cong H_4(Y_1;\Z)\oplus H_4(Y_2;\Z)\oplus H_2(Y_1;\Z)\otimes H_2(Y_2;\Z) \cong \Z/4 \oplus \Z/4 \oplus \Z/2.\]
We display the relevant groups in the $E^2_{p,q} \cong H_p(Y;\Omega_q^{\Spin})$ page next, together with the differentials that we will soon have to compute. There is again no twisting in the coefficients.

\begin{center}
\begin{tikzpicture}[scale=0.935]

\draw[step=2.0,black,xshift=1cm,yshift=1cm] (0,0) grid (12.5, 5.5);

\draw (1,2) -- (13.5,2);
\draw (1,4) -- (13.5,4);
\draw (1,6) -- (13.5,6);

\node at (0.5,1.5) {$0$};
\node at (0.5,2.5) {$1$};
\node at (0.5,3.5) {$2$};
\node at (0.5,4.5) {$3$};
\node at (0.5,5.5) {$4$};
\node at (0.5,6.5) {$q$};

\node at (2,0.5) {$0$};
\node at (4,0.5) {$1$};
\node at (6,0.5) {$2$};
\node at (8,0.5) {$3$};
\node at (10,0.5) {$4$};
\node at (12,0.5) {$5$};
\node at (14,0.5) {$p$};

\node at (10,1.7) {$(\Z/4)^2 \oplus$};
\node at (10,1.2) {$\Z/2$};
\node at (12,1.5) {$H_5(Y;\Z)$};

\node at (6,2.5) {$\Z/2 \oplus \Z/2$};
\node at (8,2.5) {$\Z/2 \oplus \Z/2$};
\node at (10,2.5) {$H_4(Y;\Z/2)$};

\node at (4,3.5) {$0$};
\node at (6,3.5) {$\Z/2 \oplus \Z/2$};

\node at (2,4.5) {$0$};
\node at (4,4.5) {$0$};

\node at (2,5.5) {$16\Z$};


\draw [thick, -latex](11.12,1.8) -- (8.88,2.4);
\draw [thick, -latex](9.12,2.72) -- (6.88,3.28);
\draw [thick, -latex](9.12,1.72) -- (6.88,2.28);
\end{tikzpicture}
\end{center}
    We compute the differential \[d_2 \colon E^2_{4,0} \cong \Z/4 \oplus \Z/4 \oplus \Z/2 \to E^2_{2,1} \cong \Z/2 \oplus \Z/2.\]
    By \cref{prop:differentials}~\eqref{prop:differentials-item-ii}, this is given as the composition
    \[H_4(Y;\Z) \to H_4(Y;\Z/2) \to H_2(Y;\Z/2)\]
    of reduction of coefficients modulo two, followed by the dual to the map $\Sq^2 + (x_1 +x_2)\cup- \colon H^2(Y;\Z/2) \to H^4(Y;\Z/2)$.
    By the Künneth theorem and  \cref{prop:homology-K2-low-degrees}, we have that $H_4(Y;\Z/2) \cong \Z/2 \oplus \Z/2 \oplus \Z/2$, and the reduction modulo two map \[H_4(Y;\Z) \cong \Z/4 \oplus \Z/4 \oplus \Z/2 \to H_4(Y;\Z/2) \cong \Z/2 \oplus \Z/2 \oplus \Z/2\] is the projection.
    To compute the second map in the composition, note that for both $x_i$, we have $\Sq^2(x_i)+(x_1+x_2)x_i=x_1x_2\in H^4(Y;\Z/2)$.
    Combining the two maps, we see that  the kernel of $d_2\colon H_4(Y;\Z)\to H_2(Y;\Z/2)$ is $H_4(Y_1;\Z)\oplus H_4(Y_2;\Z)\cong (\Z/4)^2$ and that the $\Z/2$ summand maps to $(1,1) \in \Z/2 \oplus \Z/2$.

    We deduce that  $\Omega_4(\xi_{Y_1})\oplus \Omega_4(\xi_{Y_2})\cong \Z \oplus \Z$ surjects onto $\Omega_4(\xi_{Y})$. To see this first note that the factor inclusion $Y_i \to Y$ does indeed induce a map $\Omega_4(\xi_{Y_i}) \to \Omega_4(\xi_{Y})$. In addition, other than the one we just analysed, there are no nontrivial differentials with domain the terms $E^k_{0,4}$, $E^k_{2,2}$, $E^k_{3,1}$, or $E^k_{4,0}$, for any $k \geq 2$.  The corresponding $E^{\infty}$ pages are therefore quotients of two copies of the terms on the $p+q=4$ line of the first displayed $E^2$ page above (the one from the James spectral sequence for $\Omega_4(\xi_{Y_i})$).

	The element $16\in 16\Z \leq \Z\cong \Omega_4(\xi_{Y_i})$ is represented by $K3$ for both $i$. Hence considering its image in $\Omega_4(\xi_{Y})$ we see that $(16,0)=(0,16) \in \Omega_4(\xi_{Y})$. Hence certainly $(16,-16)$ maps to zero in $\Omega_4(\xi_{Y})$. Since the signature gives a nontrivial homomorphism $\Omega_4(\xi_{Y}) \to \Z$, and there is one relation already, we  know $\Omega_4(\xi_{Y})$ has rank one.
    It follows that there exists some positive integer $\ell$ dividing $16$ such that $\Omega_4(\xi_Y)\cong \Z^2/\{(n\ell,-n\ell)\mid n\in \Z\}$.
	
	Since for both $x_i$, we have $\Sq^2(x_i)+(x_1+x_2)x_i=x_1x_2\in H^4(Y;\Z/2)$, the element $(1,1)\in H_2(Y;\Z/2) \cong \Z/2 \oplus \Z/2$ lies in the image of \[d_2\colon E^2_{4,1} \cong  H_4(Y;\Z/2)\to E^2_{2,2} \cong H_2(Y;\Z/2)\] by \cref{prop:differentials}~\eqref{prop:differentials-item-i}; this is essentially the same computation as that performed above to compute part of $d_2 \colon E^2_{4,0} \to E^2_{2,1}$. So $\ell \leq 8$.  There are no nontrivial differentials into $H_4(Y;\Z)$, and no further differentials out of it other than the one we computed. Thus $\ell \geq 4$. Hence $\ell=4$ or $\ell=8$, depending on whether the image of   \[d_2\colon E^2_{5,0} \cong H_5(Y;\Z)\to E^2_{3,1} \cong H_3(Y;\Z/2)\]
is nontrivial or trivial respectively.
	Let us  compute this differential. First note that
    \[H_5(Y;\Z)\cong H_5(Y_1;\Z)\oplus H_5(Y_2;\Z)\oplus \Tor_1(H_2(Y_1;\Z),H_2(Y_2;\Z))\]
    by the Künneth theorem. As shown in the proof of \cite{teichner-phd}*{Theorem~3.3.2~(2)}, all differentials vanishes on the $H_5(Y_i;\Z)$-summands. It remains to compute the differential on $\Tor_1(H_2(Y_1;\Z),H_2(Y_2;\Z)) \cong \Z/2$.

    We derive bases for the $\Z/2$-coefficient cohomology and homology of $Y$ using the Künneth theorem and \cref{prop:homology-K2-low-degrees}.
    For $H^3(Y;\Z/2)$, we consider the basis $\{\Sq^1(x_1),\Sq^1(x_2)\}$ and for $H^5(Y;\Z/2)$ we consider the basis	
	\[\{\Sq^2\Sq^1(x_1),x_1\Sq^1(x_1),x_2\Sq^1(x_1),x_1\Sq^1(x_2),
	x_2\Sq^1(x_2),\Sq^2\Sq^1(x_2)\}.\] For $H_3(Y;\Z/2)$ and $H_5(Y;\Z/2)$ we consider the dual bases, denoted e.g.\ $\{\Sq^1(x_1)^{\wedge},\Sq^1(x_2)^\wedge\}$.

    Again by \cref{prop:differentials}~\eqref{prop:differentials-item-ii}, we need to understand the effect of the reduction modulo two map $H_5(Y;\Z) \to H_5(Y;\Z/2)$ on the $\Tor$ summand, followed by the dual to $\Sq^2(-) + (x_1 + x_2)\cup -$.
    Since the generator of $\Z/2\cong \Tor_1(H_2(Y_1;\Z),H_2(Y_2;\Z))\leq H_5(Y;\Z)$ lies in the image of induced map on homology from $Y_1^{(3)}\times Y_2^{(3)}$, and  is invariant under the action of the automorphism interchanging $Y_1$ and $Y_2$, it follows that its image in $H_5(Y;\Z/2)$ under reduction modulo two is $(x_2\Sq^1(x_1))^{\wedge}+(x_1\Sq^1(x_2))^{\wedge}$.
	
	For the generators $\Sq^1(x_i)$ of $H^3(Y;\Z/2)$, we have \[\Sq^2\Sq^1(x_i)+(x_1+x_2)\Sq^1(x_i)=\Sq^2\Sq^1(x_i)+x_1\Sq^1(x_i)+x_2\Sq^1(x_i)\in H^5(Y;\Z/2).\]
	Computing the dual of this  map and applying it to $(x_2\Sq^1(x_1))^{\wedge}+(x_1\Sq^1(x_2))^{\wedge}$ shows that the image of $d_2\colon H_5(Y;\Z)\to H_3(Y;\Z/2) \cong \Z/2 \oplus \Z/2$ is the subgroup generated by $\Sq^1(x_1)^{\wedge}+\Sq^1(x_2)^{\wedge}$. Hence this differential is nontrivial, so $\ell=4$, and \eqref{it:wtF-ii} follows.

    Part \eqref{it:wtF-iii} follows from \cref{cor:cohom-wtF}.

Now we show \eqref{it:wtF-iv}.
    Since $H_2(\wt{F}^{(2)};\Z/2)\to H_2(\wt{F};\Z/2)$ is surjective, the image of $\Omega_4(\xi_{\wt{F}^{(2)}})$ in $\Omega_4(\xi_{\wt{F}})$ is the filtration step $F_{2,2}$ in the spectral sequence for the latter group.
    By the proof of \eqref{it:wtF-ii}, the filtration arising from the spectral sequence is
    \begin{equation}\label{eqn:filtration-groups}
        0 \leq 16\Z \cong \frac{16\Z \oplus 16\Z}{(16,-16)} \leq \frac{8\Z \oplus 8\Z}{(8,-8)} \leq \frac{4\Z \oplus 4\Z}{(4,-4)} \leq \frac{\Z \oplus \Z}{(4,-4)} = \Omega_4(\xi_{\wt{F}})
         \end{equation}
    Hence in particular $F_{2,2} \cong (8\Z\oplus8\Z)/(8,-8)$, which proves \eqref{it:wtF-iv}.

  Finally, we showed in the proof of \eqref{it:wtF-ii} that the kernel of $d_2\colon H_5(Y;\Z)\to H_2(Y;\Z/2)$ is $H_5(Y_1;\Z)\oplus H_5(Y_2;\Z)$.
   As mentioned above, it was shown in the proof of \cite{teichner-phd}*{Theorem~3.3.2~(2)} that all differentials, including the $d_3$ differentials, vanish on the $H_5(Y_i;\Z)$-summands. This shows \eqref{it:wtF-v}.
\end{proof}

The next lemma implies that $[K3]=0\in \Omega_4(\xi_F)$, in particular proving \cref{thm:vanishing-K3-in-Omega-4-F}. We will need the slightly stronger statement given for our computation of the $d_4$ differential later.

\begin{lemma}
	\label{lem:d4-nonzero}
	Let $j\colon F^{(4)}\to F$ be the inclusion of the $4$-skeleton. Let $\xi_{F^{(4)}} \colon B_{F^{(4)}} \to \BO$ be the fibration determined by $(F^{(4)},j^*v_1,j^*v_2)$. Then $[K3]=0\in \Omega_4(\xi_{F^{(4)}})$.
\end{lemma}

\begin{proof}
    Choose a CW structure on $F$ and let $\wt F^{(2)}$ and $\wt F^{(4)}$ denote the corresponding $2$- and $4$-skeleta of the universal cover, respectively. Let $k\colon \wt F^{(2)} \to \wt F^{(4)}$ denote the inclusion map.
    Consider the map $j_* \colon \Omega_4(\xi_{\wt{F}^{(4)}})\to \Omega_4(\xi_{\wt{F}})$, and let $j_*|$ denote its restriction to the image of $\Omega_4(\xi_{\wt{F}^{(2)}})$ in domain and codomain.
   It follows from the vanishing of
    $d_3\colon E_{5,0}^3\to E_{2,2}^3$, shown in \cref{lem:wtF-bordism} \eqref{it:wtF-v}, that $j_*|$ is an isomorphism.
By  \cref{lem:wtF-bordism} \eqref{it:wtF-iv},  \[\im \Omega_4(\xi_{\wt{F}^{(2)}}) \cong (8\Z\oplus8\Z)/(8,-8).\]

The class $[K3]$ lies in the image of $\Omega_4(\xi_{\wt{F}^{(2)}})$ and is represented by $(16,0)$. By \cref{cor:wtF,lem:wtF-bordism}, the action of $\Z/2=\pi_1(F)$ on $\Omega_4(\xi_{\wt{F}}) \cong \Omega_4(\xi_{Y})\cong \Z^2/(4,-4)$ interchanges the summands. Moreover, since $v_1$ is nontrivial on the generator of $\pi_1(F)$, acting by this generator changes the orientation. So altogether the action sends $(z,z')\mapsto (-z',-z)$. The deck transformation of $\wt F$ restricts to $\wt{F}^{(2)}$ and $\wt{F}^{(4)}$. Hence on the image $(8\Z\oplus8\Z)/(8,-8)$ of $\Omega_4(\xi_{\wt{F}^{(2)}})$ in $\Omega_4(\xi_{\wt{F}^{(4)}})$ the deck transformation acts by sending $(-8,0)$ to $(0,8)$.

    We consider the composition $\wt F^{(2)}\xrightarrow{k} \wt F^{(4)}\xrightarrow{p} F^{(4)}$ and the induced maps on bordism groups
    \[\Omega_4(\xi_{\wt{F}^{(2)}})\xrightarrow{k_*} \Omega_4(\xi_{\wt{F}^{(4)}})\xrightarrow{p_*} \Omega_4(\xi_{F^{(4)}}).\]
    In particular this composition factors through $p_*| \colon \im k_* \to \Omega_4(\xi_{F^{(4)}})$.
    Since  deck transformations commute with $p \colon \wt{F}^{(4)} \to F^{(4)}$, $p_*|$  factors through the quotient of $\im k_*$ by the $\Z/2$ deck transformation action, yielding
    \[(8\Z\oplus8\Z)/(8,-8) \cong \im k_* \to \Z \otimes_{\Z[\Z/2]} \im k_* \to \Omega_4(\xi_{F^{(4)}}).\]
  The class $[K3]$ is represented by $(16,0) \in (8\Z\oplus8\Z)/(8,-8)$,
     but in the quotient of $(8\Z\oplus8\Z)/(8,-8)$ by the action sending $(-8,0)$ to $(0,8)$ we have
    \[(16,0)=(8,8)=(8,0)+(0,8)=(8,0)+(-8,0)=(0,0),\]
    where the first equality is given by subtracting the trivial element $(8,-8)$, and the penultimate equality uses the deck transformation.
     Hence $[K3]$  vanishes in $\Omega_4(\xi_{F^{(4)}})$,  as claimed.
\end{proof}

\begin{corollary}\label{cor:d4-nontrivial-F}
    The  differential $d_4 \colon H_4(F;\Z/2) \to H_0(F;\Omega_4^{\Spin})$ in the James spectral sequence computing $\Omega_4(\xi_F)$ is nontrivial.
\end{corollary}

\begin{proof}
    By \cref{thm:d3-diff}, the differential $d_3 \colon H_3(F^{(4)};\Z/2) \to H_0(F^{(4)};\Omega_4^{\Spin})=E^3_{0,4}\cong \Z/2$ in the James spectral sequence for $F^{(4)}$ is dual to $j^*v_1^3+j^*v_1j^*v_2=j^*(v_1^3+v_1v_2)=0$. Hence $E_{0,4}^4 \cong \Z/2$.
    By \cref{lem:d4-nonzero}, $[K3]$ vanishes in $\Omega_4(\xi_{F^{(4)}})$ and hence $E_{0,4}^\infty=0$.
    Since $H_5(F^{(4)};\Z)=0$, it follows that the differential $d_4 \colon H_4(F^{(4)};\Z/2) \to H_0(F^{(4)};\Omega_4^{\Spin})$ must be nontrivial. Since $j_*\colon H_4(F^{(4)};\Z/2)\to H_4(F;\Z/2)$ is surjective, the differential $d_4 \colon H_4(F;\Z/2) \to H_0(F;\Omega_4^{\Spin})$ is also nontrivial by naturality of the James spectral sequence.
\end{proof}

\section{Computing the \texorpdfstring{$d_4$}{d4} differential}\label{section:determining-d-4-diff}

Now we know that there is a nonzero $d_4$ differential in the spectral sequence for the universal space $F$ corresponding to $w_1^3 = w_1w_2$ from \eqref{eqn:defn-of-F}, our next aim is to compute this differential.

  Let $V$ be a space and let $w_i\in H^i(V;\Z/2)$ for $i=1,2$.
Define
\[\Sq^2_{w_1,w_2}:= \Sq^2(-)+w_1\Sq^1(-)+w_2\cup -\colon H^2(V,\Z/2)\to H^4(V;\Z/2),\]
and let
\[\Sq_2^{w_1,w_2}\colon H_4(V;\Z/2)\to H_2(V;\Z/2)\] denote its dual.

\begin{lemma}
	\label{lem:trivial-differential}
For $v_1$ and $v_2$ as in \eqref{defn:v-i},	the map
	\[\Sq^2_{v_1,v_2} \colon H^2(F;\Z/2)\to H^4(F;\Z/2)\]
	is trivial.
\end{lemma}

\begin{proof}
	By \cref{lem:cohom-classes-F}, $H^2(F;\Z/2)$ is generated by $v_1^2$ and $v_2$.
	Since $v_1^3=v_1v_2$,
	\[v_1^4=\Sq^1(v_1^3)=\Sq^1(v_1v_2)=v_1^2v_2+v_1\Sq^1(v_2)=v_1^4+v_1\Sq^1(v_2).\]
	Thus $v_1\Sq^1(v_2)=0$. Using this and again that $v_1^3=v_1v_2$, we have
	\[\Sq^2(v_1^2)+v_1\Sq^1(v_1^2)+v_2v_1^2=v_1^4+0+v_1^2v_2= v_1(v_1^3 + v_1v_2) =0,\]
	and
	\[\Sq^2(v_2)+v_1\Sq^1(v_2)+v_2^2=v_1\Sq^1(v_2)=0.\qedhere\]
\end{proof}

By \cref{prop:differentials} and \cref{lem:trivial-differential}, the differential  $d_2 \colon H_4(F;\Omega_1^{\spin}) \to H_2(F;\Omega_2^{\Spin})$ is trivial. Hence the $d_4$ differential $d_4 \colon E^4_{4,1} \to E^4_{0,4}$ has domain $H_4(F;\Omega_1^{\spin})/\im d_2^F$, where $d_2^F \colon E_{6,0}^2\to E_{4,1}^2$.

\begin{definition}\label{defn:mathfrak-o}
    We define $\mathfrak{o}\in H^4(F;\Z/2)$ to be the cohomology class dual, in the sense of \cref{thm:d3-diff}, to
	\[H_4(F;\Z/2)\twoheadrightarrow H_4(F;\Z/2)/\im d_2^F\xrightarrow{d_4^F}H_0(F;\Omega_4^{\spin})\cong \Z/2.\]
    Equivalently, this composition yields $\mathfrak{o}$ under the identification $\Hom(H_4(F;\Z/2),\Z/2) \cong H^4(F;\Z/2)$.
\end{definition}

\begin{proposition}
	\label{prop:choice-of-lift}
     A triple $(V,w_1,w_2)$ determines a fibration $\xi_V \colon B_V \to \BO$ as in \cref{defn:key-pullback}, with a corresponding James spectral sequence for $\Omega_4(\xi_V)$.  Suppose that $w_1^3=w_1w_2\in H^3(V;\Z/2)$.
    Let $f\colon V \to F$ be a lift of $(w_1,w_2)\colon V \to X$ along $u \colon F \to X$.
	The class
	\[[f^*\mathfrak{o}]\in H^4(V;\Z/2)/\im(\Sq^2_{w_1,w_2} \colon H^2(V,\Z/2)\to H^4(V;\Z/2))\]
	is dual to the composition
	\[\ker(\Sq_2^{w_1,w_2})\twoheadrightarrow \ker(\Sq_2^{w_1,w_2})/\im d_2^V\xrightarrow{d_4^V}H_0(V;\Omega_4^{\spin})\twoheadrightarrow \Z/2.\]
	In particular, the class $[f^*\mathfrak{o}]$ is independent of the choice of lift $f$.
\end{proposition}

\begin{proof}
	 The lift $f$ induces a map $B_V \to B_F$ over $\BO$ and a corresponding map of bordism groups $\Omega_4(\xi_V) \to \Omega_4(\xi_F)$.
	
	By \cref{prop:differentials} and \cref{lem:trivial-differential}, the $d_2$ differential with domain  $H_4(F;\Omega_1^{\spin})$ is trivial. The differential with domain $H_4(V;\Omega_1^{\spin})$ is given by $\Sq_2^{w_1,w_2}$. Hence the $E^4_{4,1}$ term for $F$ is $H_4(F;\Z/2)/\im d_2^F$ and the $E^4_{4,1}$ term for $V$ is $\ker(\Sq_2^{w_1,w_2})/\im d_2^V$ (as before recall that $\Omega_3^{\Spin}=0$ so in both cases $d_3=0$).
	By naturality of the James spectral sequence, there is a commutative diagram
	\[\begin{tikzcd}
		\ker(\Sq_2^{w_1,w_2})\ar[r, two heads]\ar[d,"f_*"]& \ker(\Sq_2^{w_1,w_2})/\im d_2^V\ar[r,"{d_4^V}"]\ar[d,"f_*"] & H_0(V;\Omega_4^{\spin})\ar[d,two heads]\\
		H_4(F;\Z/2)\ar[r, two heads] \ar[rr,"\mathfrak{o}",bend right=15] & H_4(F;\Z/2)/\im d_2^F\ar[r,"{d_4^F}"]&H_0(F;\Omega_4^{\spin})\cong \Z/2.
	\end{tikzcd}\]
	By \cref{defn:mathfrak-o}, the bottom composition is dual to $\mathfrak{o}$. Hence the down-right-right composition is dual to $[f^*\mathfrak{o}]$. Thus the right-right-down composition is also dual to $[f^*\mathfrak{o}]$, as claimed. Since the differential $d_4^V$ is independent of the choice of lift $f$, so is the class $[f^*\mathfrak{o}]$.
\end{proof}

Recall that $p \colon \wt{F} \to F$ denotes the projection map of the universal cover.

\begin{lemma}
	\label{cor:d4}
    The class $\mathfrak{o}\in H^4(F;\Z/2)$ is
    the unique  class $\mathfrak{o}\in H^4(F;\Z/2)$ such that $p^*\mathfrak{o}=x_1x_2 \in H^4(\wt{F};\Z/2)$ and $s^*\mathfrak{o}=0 \in H^4(K;\Z/2)$ for any choice of  section $s\colon K_1\to F$ of $c \colon F \to K_1$.
\end{lemma}
\begin{proof}
    From \cref{cor:wtF}, \cref{prop:homology-K2-low-degrees}, and the computations with the K\"unneth theorem in the proof of \cref{lem:wtF-bordism},
    it follows that \[H^2(\wt F;\Z/2)\cong (\Z/2)[\Z/2],\, H^3(\wt F;\Z/2)\cong (\Z/2)[\Z/2], \text{ and }H^4(\wt F;\Z/2)\cong (\Z/2)[\Z/2]\oplus \Z/2.\]
    Since $H^p(K_1;(\Z/2)[\Z/2]) \cong H^p(S^{\infty};\Z/2)=0$ for $p>0$, the Leray--Serre spectral sequence for $\wt F\to F\to K_1$ shows that
    there is a short exact sequence
    \begin{equation}\label{eqn:ses-for-h4-F-Z2}
        0\to H^4(K_1;\Z/2)\xrightarrow{c^*} H^4(F;\Z/2)\xrightarrow{p^*}H^0(K_1;H^4(\wt F;\Z/2))\to 0.
        \end{equation}
    We have $H^0(K_1;H^4(\wt F;\Z/2))\cong H^4(\wt F;\Z/2)^{\Z/2}\cong (\Z/2)^2$ generated by $x_1x_2$ and $x_1^2+x_2^2$.
    Here recall that the $x_i$ denote the generators of $H^2(Y;\Z/2) = H^2(\wt{F};\Z/2)$.
    Furthermore, a section $s\colon K_1\to F$ of $c$ yields a splitting of the short sequence above.
    Thus a basis of $H^4(F;\Z/2)$ is given by $b:=c^*(\iota_1^4)$, the image of $H^4(K_1;\Z/2)$, and $a_1$ and $a_2$, where $a_1$ and $a_2$ are determined by $p^*(a_1)=x_1^2+x_2^2$, $p^*(a_2)=x_1x_2$, and the condition that $s^*(a_i)=0$. In particular $\mathfrak{o}$ is a $\Z/2$-linear combination of $a_1$, $a_2$, and $b$.
It also follows from the short exact sequence~\eqref{eqn:ses-for-h4-F-Z2} that the conditions $p^*\mathfrak{o}=x_1x_2$ and $s^*\mathfrak{o}=0$ determine a unique class.

    Let $s\colon K_1\to F$ be a section. We now show that $s^*\mathfrak{o}=0$.
	We compare the James spectral sequence for $\Omega_4(\xi_F)$ with the spectral sequence for $\Omega_4(\xi_{K_1})$, where  $\xi_{K_1}$ is the fibration determined by $(K_1,s^*v_1,s^*v_2)$ as in \cref{defn:key-pullback}, and the $v_i$ are as in \eqref{defn:v-i}.
    Since $s$ is a section, naturality of the James spectral sequences implies that the $d_2$ differential with domain $H_4(K_1;\Omega_1^{\spin})$ is again trivial.
    Therefore $E^4_{4,1}\cong E^2_{4,1} \cong H_4(K_1;\Z/2)$ (there is no $d_3$ either, because $\Omega_3^{\Spin}=0$).
	Also recall from \cref{lem:cohom-classes-F} that $s^*v_1^2=(s^*v_1)^2 =s^*v_2$. We can therefore apply \cref{thm:kreck}, which implies that the differential $d_4^{K_1}\colon E^4_{4,1}\cong H_4(K_1;\Z/2)\to E^4_{0,4}$ is trivial. By naturality of the James spectral sequence the following diagram commutes,
 \[\begin{tikzcd}[column sep=large]
        H_4(K_1;\Z/2)\ar[r,"\cong"]\ar[d,"s_*"]&E_{4,1}^4\ar[r,"d_4^{K_1}=0"]\ar[d,"s_*"]&E^4_{0,4}\cong \Z/2\ar[d,"\cong","s_*"']\\
        H_4(F;\Z/2)\ar[r,two heads]\ar[rr,"\mathfrak{o}",bend right=15]&E_{4,1}^4\ar[r,"d_4^F"]&E^4_{0,4}\cong \Z/2
    \end{tikzcd}\]
This implies that $s^*\mathfrak{o}=0$, as desired.

It remains to show that $p^*\mathfrak{o}=x_1x_2$, i.e.\ that $\mathfrak{o}=a_2$.
Since $s$ is a section of $c$, $s^*b = s^*c^*(\iota_1^4) = \iota_1^4 \neq 0$.  Thus since $s^*\mathfrak{o}=0$, we see that $\mathfrak{o}=\varepsilon_1a_1+\varepsilon_2a_2$ for some $\varepsilon_i\in \Z/2$. Since the $d_4$ differential is trivial for $\Omega_4(\xi_{\wt F})$ (the domain is 2-torsion and the codomain is $16\Z$), $p^*\mathfrak{o}=p^*(\varepsilon_1a_1+\varepsilon_2a_2)$ has to be trivial in $H^*(\wt F;\Z/2)/\im(\Sq^2_
    {w_1,w_2})$, by \cref{prop:choice-of-lift} applied with $f=p$. That is, $p^*(\varepsilon_1a_1+\varepsilon_2a_2)$ lies
    in the image of the map $\Sq^2_{w_1,w_2}\colon H^2(\wt F;\Z/2)\to H^4(\wt F;\Z/2)$. This image is $\{0,x_1x_2\}$ and hence $\varepsilon_1=0$. Since $\mathfrak{o}\neq 0$ by \cref{cor:d4-nontrivial-F}, $\mathfrak{o}=a_2$ as claimed.
\end{proof}

\begin{proof}[Proof of \cref{thm:B}]
	Item \eqref{it:B-1} was shown in \cref{prop:hom-groups-F} and \eqref{it:B-2} follows from \cref{cor:d4}.
	
	Now we show \eqref{it:B-4}. Let $(\pi,w_1,w_2)$ be a normal $1$-type. If $w_1^3=w_1w_2$, there is a lift $f\colon B\pi \to F$ of $(w_1,w_2)$ along $u$. If
	\[0=[f^*\mathfrak{o}]\in H^4(\pi;\Z/2)/\im\big(\Sq^2_{w_1,w_2} \colon H^2(\pi,\Z/2)\to H^4(\pi;\Z/2)\big),\]
    then the composition
    \begin{equation}\label{eqn:composition-with-d4}
        \ker(\Sq_2^{w_1,w_2})\twoheadrightarrow \ker(\Sq_2^{w_1,w_2})/\im d_2^\pi\xrightarrow{d_4^\pi} \Z/2
    \end{equation}
	vanishes, by \cref{prop:choice-of-lift} applied with $V=B\pi$. Hence $d_4^\pi=0$ because the first map is surjective. It follows that $[K3]$ survives to the $E^5$-page. By hypothesis $H_5(\pi;\Z)=0$, and so in fact $[K3] \neq 0\in \Omega_4(\xi_{\pi})$. The existence of stably exotic $4$-manifolds with normal $1$-type $(\pi,w_1,w_2)$ now follows from \cref{prop:strategy}~\eqref{item:iii}. This completes the proof of \eqref{it:B-4}.
	
	Lastly, we show \eqref{it:B-3}. Let $(\pi,w_1,w_2)$ be a normal $1$-type. If $w_1^3\neq w_1w_2$, no stably exotic $4$-manifolds with normal $1$-type $(\pi,w_1,w_2)$ exist by \cref{thm:main}~\eqref{it:main-1}. So we assume $w_1^3=w_1w_2$. As above, there is a lift $f\colon B\pi \to F$ of $(w_1,w_2)$ along $u$. If
	\[0\neq [f^*\mathfrak{o}]\in H^4(\pi;\Z/2)/\im \Sq^2_{w_1,w_2},\]
	then the composition \eqref{eqn:composition-with-d4} is nonzero by \cref{prop:choice-of-lift}, again applied with $V=B\pi$. Hence $d_4^\pi \neq 0$, and so  $0=[K3]\in \Omega_4(\xi_{\pi})$. Thus no stably exotic $4$-manifolds with normal $1$-type $(\pi,w_1,w_2)$ exist by \cref{prop:strategy}~\eqref{item:ii}. This implies \eqref{it:B-3}.
\end{proof}

We close the section by giving the details of the example promised in \cref{remark-in-intro-example}. This is an example where $w_1^3=w_1w_2$ and the obstruction $f^*\mathfrak{o}$ is nontrivial.

\begin{proposition}
\label{prop:ex1.5}
    Let $a_k$ with $k\in \Z/4$ be generators of $\Z^4$. Let $\pi\cong \Z^4\rtimes \Z/2$, where $\Z/2$ acts on $\Z^4$ by sending $a_k\to a_{k+2}$. There exist classes $w_i\in H^i(\pi;\Z/2)$ for $i=1,2$ such that $w_1^3=w_1w_2$ but no stably exotic $4$-manifolds with normal $1$-type $(\pi,w_1,w_2)$ exist.
\end{proposition}

\begin{proof}	
    First we construct a map $f\colon B\pi\to F$ as follows. Start with a map $\wh f\colon T^4\to \wt F\simeq K_2\times K_2$ given by $(t_1t_2,t_3t_4)$, where $t_k\in H^1(\Z^4,\Z/2)$ corresponds to the generator $a_k$ of $\Z^4$. This map is equivariant with respect to the $\Z/2$-actions by \cref{cor:wtF} and hence descends to a map $f\colon B\pi\to F$. Let $w_i:=f^*v_i$ for $i=1,2$, where $v_1$ and $v_2$ are as in \eqref{defn:v-i}.   Then $w_1^3=w_1w_2$ holds by construction of $F$.

     By \cref{thm:B}~\eqref{it:B-3}, it remains to show that
     \[[f^*\mathfrak{o}]\neq 0\in H^4(\pi;\Z/2)/\im \Sq^2_{w_1,w_2},\] i.e.\ that $f^*\mathfrak{o}$ does not lie in the image of \[\Sq^2_{w_1,w_2} = \Sq^2(-)+w_1\Sq^1(-)+w_2 \cup - \colon H^2(\pi,\Z/2)\to H^4(\pi;\Z/2).\]
   To begin the proof of this, note that by construction $\wh{f}^*x_1=t_1t_2$ and $\wh{f}^*x_2=t_3t_4$. By \cref{thm:B} \eqref{it:B-2}, $p^*\mathfrak{o}=x_1x_2$.
    Thus
    \begin{equation}\label{eqn:wh-f-p-o-nontrivial}
        \wh{f}^*p^*\mathfrak{o}=\wh f^*(x_1x_2)=t_1t_2t_3t_4\neq 0\in H^4(\Z^4;\Z/2).
    \end{equation}
    Let $j\colon T^4\to B\pi$ be determined by the inclusion of $\Z^4$.
    Note that $f \circ j =  p \circ \wh{f} \colon T^4 \to F$.
    Since $j^*f^*\mathfrak{o}=\wh f^*p^*\mathfrak{o}\neq 0$ by \eqref{eqn:wh-f-p-o-nontrivial}, if we show that $j^*\circ \Sq^2_{w_1,w_2}=0$ it will follow that $f^*\mathfrak{o}$ does not lie in the image of $\Sq^2_{w_1,w_2}$. Since $j^*w_1=0$, $j^*w_2=\wh f^*p^*v_2=\wh f^*(x_1+x_2)=t_1t_2+t_3t_4$ and $\Sq^2$ vanishes on $T^4$, the map $j^*\circ \Sq^2_{w_1,w_2}$ is given by \[j^*(w_2 \cup -) = j^*w_2 \cup j^*(-) = (t_1t_2+t_3t_4) \cup j^*(-).\]
    Furthermore the image of $j^*\colon H^2(\pi;\Z/2)\to H^2(\Z^4;\Z/2)$ consists of the elements invariant under the $\Z/2$-action. Hence a basis for $\im j^*$ is given by $\{t_1t_3,t_2t_4,t_1t_2+t_3t_4,t_1t_4+t_2t_3\}$. It is straightforward to see that the cup product of each of these elements with $t_1t_2+t_3t_4$ is trivial. Hence $j^*\circ \Sq^2_{w_1,w_2}=0$, as needed.

	Thus by \cref{thm:B}~\eqref{it:B-3}, there are no stably exotic $4$-manifolds with normal $1$-type $(\pi,w_1,w_2)$.
\end{proof}

\section{Computing the \texorpdfstring{$d_5$}{d5} differential}\label{section:determining-d-5-diff}

Our next goal is to understand the class $[K3]$ in the case that $w_1^3 + w_1w_2 =0$ and the pullback of $\mathfrak{o}$ from \cref{cor:d4} is also trivial. For this we consider the fibration sequence
\begin{equation}\label{eqn:universal-fibration-for-d5}
G \xrightarrow{y} F \xrightarrow{\mathfrak{o}}K(\Z/2,4).\end{equation}

\begin{lemma}\leavevmode
	\label{lem:homology-G}
	\begin{enumerate}[(a)]
		\item\label{it:hom-G-i} The map $\wt{y}_*\colon H_i(\wt G;\Z/2)\to H_i(\wt F;\Z/2)$ is an isomorphism for $i\leq 3$.
		\item\label{it:hom-G-ii} The map $\wt{y}_*\colon H_5(\wt G;\Z)\to H_5(\wt F;\Z)$ is surjective.
	\end{enumerate}
\end{lemma}

\begin{proof}
    We first compute the map on $\Z/2$-homology.
	Consider the Leray--Serre spectral sequence for the fibration $K(\Z/2,3)\to \wt G\to \wt F$ obtained by looping the fibration \eqref{eqn:universal-fibration-for-d5}.
    Since $H_q(K(\Z/2,3);\Z/2)=0$ for $q=1,2$, the map $\wt{y}_*\colon H_i(\wt G;\Z/2)\to H_i(\wt F;\Z/2)$ is an isomorphism for $i\leq 2$ and there is an exact sequence
	\begin{equation}\label{eqn:an-exact-sequence-with-wtF-and-wtF-prime}
  \cdots\to H_4(\wt G;\Z/2)\to H_4(\wt F;\Z/2)\xrightarrow{d_4} H_0(\wt F;H_3(K(\Z/2,3);\Z/2))\to H_3(\wt G;\Z/2)\xrightarrow{\wt{y}_*} H_3(\wt F;\Z/2)\to 0.
  \end{equation}
We will show that $d_4$ is nontrivial.
The concatenation of the two pullback squares
\[\begin{tikzcd}
    \wt{G} \ar[r] \ar[d] & G \ar[r] \ar[d] & \{*\} \ar[d] \\
    \wt{F} \ar[r,"p"] & F \ar[r,"\mathfrak{o}"] & K(\Z/2,4)
  \end{tikzcd}\]
is again a pullback square, showing that $\wt G$ is the homotopy fibre of $p^*\mathfrak{o} \colon \wt F\to K(\Z/2,4)$.  It follows that the composition $H_4(\wt G;\Z/2)\to H_4(\wt F;\Z/2)\to H_4(K(\Z/2,4);\Z/2)$ is trivial.  Since $p^*\mathfrak{o}=x_1x_2\neq 0$ by \cref{cor:d4}, the map $H_4(\wt F;\Z/2)\to H_4(K(\Z/2,4);\Z/2)$ is nontrivial. Thus  $H_4(\wt G;\Z/2)\to H_4(\wt F;\Z/2)$ is not surjective, and so by exactness of \eqref{eqn:an-exact-sequence-with-wtF-and-wtF-prime} the map $d_4$ is nontrivial, as desired.
	However $H_3(K(\Z/2,3);\Z/2)\cong \Z/2$, and so the codomain of $d_4$ is also $\Z/2$, and it follows that $d_4$ is surjective. Hence by exactness of \eqref{eqn:an-exact-sequence-with-wtF-and-wtF-prime}, $\wt{y}_*\colon H_3(\wt G;\Z/2)\to H_3(\wt F;\Z/2)$ is an isomorphism. This shows~\eqref{it:hom-G-i}.

	Now consider the integral Leray--Serre spectral sequence for the same fibration $K(\Z/2,3)\to \wt G\to \wt F$.
	Since $H_q(K(\Z/2,3);\Z)=0$ for $q=1,2,4$ and
	$H_1(\wt F;H_3(K(\Z/2,3);\Z))=0$, there are no nontrivial differentials out of $H_5(\wt F;\Z)$. This proves~\eqref{it:hom-G-ii}.
\end{proof}

\begin{lemma}
    \label{lem:f31-G}
   For any choice of 3-skeleton $G^{(3)}$ of $G$, we consider the corresponding 3-skeleton $(\wt{G})^{(3)}$ of $\wt{G}$.
   The image of $\Omega_4(\xi_{(\wt G)^{(3)}})$ in $\Omega_4(\xi_{\wt G})$ maps isomorphically under $\wt y_*$ onto the image of $\Omega_4(\xi_{\wt F^{(3)}})$ in $\Omega_4(\xi_{\wt F})$, which is isomorphic to $(4\Z\oplus 4\Z)/(4,-4)$.
\end{lemma}

\begin{proof}
    We consider the James spectral sequences with $E^2_{p,q}$ terms $H_p(\wt G;\Omega_q^{\spin})$ and $H_p(\wt F;\Omega_q^{\spin})$, converging to $\Omega_4(\xi_{\wt G})$ and $\Omega_4(\xi_{\wt F})$, respectively. The respective images of $\Omega_4(\xi_{(\wt G)^{(3)}})$ and $\Omega_4(\xi_{\wt F^{(3)}})$ correspond to the $F_{3,1}$ filtration steps in these spectral sequences.

    Since $\wt{y}_*\colon H_i(\wt G;\Z/2)\to H_i(\wt F;\Z/2)$ is an
    isomorphism for $i\leq 3$, the terms $E^2_{0,4},E^2_{2,2}$, and $E^2_{3,1}$ agree in the spectral sequences for $\wt{G}$ and $\wt{F}$. In particular, $E^2_{3,1}\cong E^2_{2,2}\cong (\Z/2)^2$. For $\Omega_4(\xi_{\wt F})$ we have $E^\infty_{3,1}\cong E^\infty_{2,2}\cong \Z/2$ as shown in the proof of \cref{lem:wtF-bordism} \eqref{it:wtF-ii}. Hence to show that the $F_{3,1}$ filtration steps agree, it suffices to show that we also have $E^\infty_{3,1}\cong E^\infty_{2,2}\cong \Z/2$ for $\Omega_4(\xi_{\wt G})$.

    Since $p^*\mathfrak{o}=x_1x_2$ by \cref{cor:d4}, we have a fibration
    \[\wt G\to \wt F\xrightarrow{x_1x_2} K(\Z/2,4).\]
    By \cref{cor:wtF}, $\wt F\simeq Y=Y_1\times Y_2\times S^\infty$ with $Y_i=K_2$. For each $i=1,2$, one of $x_1$ and $x_2$ pulls back trivially to $Y_i$. Hence the inclusion of $Y_i$ into $\wt F\simeq Y$ lifts to $\wt G$. Thus the surjection $\Omega_4(\xi_{Y_1})\oplus \Omega_4(\xi_{Y_2}) \cong \Z \oplus \Z \to \Omega_4(\xi_{\wt F})$ from ~\eqref{it:wtF-ii} and \eqref{it:wtF-iii} of \cref{lem:wtF-bordism} factors through $\wt{y}_*$. This implies, by the same logic used in the proof of \cref{lem:wtF-bordism}~\eqref{it:wtF-ii} (but without the $E_{4,0}^{\infty} \cong \Z/4 \oplus \Z/4$ term), that the $F_{3,1}$ filtration step for $\Omega_4(\xi_{\wt G})$ is a quotient of $(4\Z\oplus 4\Z)/(16,-16)$.

    Since $\wt{y}_*\colon H_5(\wt G;\Z)\to H_5(\wt F;\Z)$ is surjective by \cref{lem:homology-G} and the differential $d_2\colon E^2_{5,0}\to E^2_{3,1}$ is nontrivial for $\Omega_4(\xi_{\wt F})$ (see the proof of \cref{lem:wtF-bordism}~\eqref{it:wtF-ii}), the same differential is nontrivial for $\Omega_4(\xi_{\wt G})$. It follows that the $E_{3,1}^\infty$ terms agree for $\wt G$ and $\wt F$ and are isomorphic to $\Z/2$.
	Hence the images of $(4,0)$ and $(0,4)$ in the $E_{3,1}^\infty$ term for $\wt G$ agree and so $(4,-4)$ lies in $F_{2,2}$. As $E_{0,4}^\infty \to F_{2,2} \to E^{\infty}_{2,2} \to 0$ is exact and $E^\infty_{2,2}$ is 2-torsion, it follows that the image of $(8,-8)$ maps trivially to $E^\infty_{2,2}$ and so lies in $E_{0,4}^\infty$. Since the images of $(8,0)$ and $(0,-8)$ generate $E_{2,2}^\infty$, we have that $E_{2,2}^\infty\cong \Z/2$ for $\wt{G}$, as for $\wt F$.
    As mentioned, the lemma follows, noting that, by the same logic as that used in \eqref{eqn:filtration-groups}, the $F_{3,1}$ filtration step for $\Omega_4(\xi_{\wt{F}})$ is isomorphic to $(4\Z\oplus 4\Z)/(4,-4)$.
\end{proof}

\begin{corollary}
	$[K3]=0\in \Omega_4(\xi_{G})$.
\end{corollary}
\begin{proof}
    As in the proof of \cref{lem:d4-nonzero}, the map \[\Omega_4(\xi_{(\wt G)^{(3)}})\to\Omega_4(\xi_{\wt G}) \to \Omega_4(\xi_{G})\] factors through $\Z\otimes_{\Z[\Z/2]}\Omega_4(\xi_{\wt G})$, where $\Z/2$ acts by the deck transformation.
    By \cref{lem:f31-G}, the image of $\Omega_4(\xi_{(\wt G)^{(3)}})$ is isomorphic to $(4\Z\oplus 4\Z)/(4,-4)$, and the deck transformation acts by $(z,z')\mapsto (-z',-z)$, as in the proof of \cref{lem:d4-nonzero}.
     We compute in the quotient that
    \[(16,0)=(8,8)=(8,0)+(0,8)=(8,0)+(-8,0)=(0,0).\]
    Since $[K3]$ represents $(16,0)$, it follows that $[K3]=0\in \Omega_4(\xi_{G})$ as claimed.
\end{proof}

As the $d_3$ and $d_4$ differentials in the James spectral sequence $H_p(G;\Omega_q^{\spin})\Rightarrow \Omega_{p+q}(\xi_{G})$ are trivial by \cref{thm:main,thm:B}, it follows that the $d_5$ differential must be nontrivial.
So to completely decide whether stably exotic 4-manifolds exist for $(\pi,w_1,w_2)$ such that $w_1^3 + w_1w_2=0$ and $f^*\mathfrak{o}=0$, one must finally consider the pullback of this universal (nontrivial) $d_5$ differential in the James spectral sequence for $\Omega_4(\xi_{\pi})$. We do not know a pleasant general way to analyse this akin to the analysis for the $d_3$ and $d_4$ differentials.

\def\MR#1{}
\bibliography{bib}
\end{document}